\documentclass[11pt]{amsart}

\usepackage{amsmath}
\usepackage{mathtools}
\usepackage{fullpage}
\usepackage{xspace}
\usepackage[psamsfonts]{amssymb}
\usepackage[latin1]{inputenc}
\usepackage{graphicx,color}
\usepackage{hyperref}
\usepackage{graphicx}
\usepackage{xcolor}

\usepackage{amsmath}%
\usepackage{amsthm}%
\usepackage{amscd}
\usepackage{amsfonts}%
\usepackage{amssymb}%
\usepackage{graphicx}
\usepackage{stmaryrd}
\usepackage{mathrsfs}
\usepackage{tikz}
\usetikzlibrary{matrix,arrows}
\usepackage{tikz-cd}
\usepackage[all]{xy}
\usepackage{bbm} 
\usepackage{longtable}

\usepackage{lineno} 

\newtheorem{theorem}{Theorem}[section]

\newtheorem{corollary}[theorem]{Corollary}

\newtheorem{example}{Example}

\newtheorem{lemma}[theorem]{Lemma}

\newtheorem{proposition}[theorem]{Proposition}
\theoremstyle{remark}
\newtheorem{remark}[theorem]{Remark}

\numberwithin{equation}{section}

\newcommand{\N}{\mathbb{N}}

\newcommand{\Q}{\mathbb{Q}}
\newcommand{\R}{\mathbb{R}}

\newcommand{\Z}{\mathbb{Z}}

\newcommand{\Hcali}{\mathcal{H}}

\newcommand{\Rcali}{\mathcal{R}}

\newcommand{\Ucali}{\mathcal{U}}

\newcommand{\Zp}{\mathbb{Z}_p}

\newcommand\rquot[2]{
  \mathchoice
  {
    \text{\raise0.5ex\hbox{$#1$}\big/\lower0.5ex\hbox{$#2$}}%
  }
  {
    #1\,/\,#2
  }
  {
    #1\,/\,#2
  }
  {
    #1\,/\,#2
  }
}

\newcommand\lrquot[3]{
  \mathchoice
  {
    \text{\lower0.5ex\hbox{$#1$}\big\backslash\raise0.5ex\hbox{$#2$\!}\big/
      \lower0.5ex\hbox{\!\!$#3$}}%
  }
  {
    #1\,\backslash\,#2\,/\,#3
  }
  {
    #1\,\backslash\,#2\,/\,#3
  }
  {
    #1\,\backslash\,#2\,/\,#3
  }
}

\newcommand\lquot[2]{
  \mathchoice
  {
    \text{\lower0.5ex\hbox{$#1$}\big\backslash\raise0.5ex\hbox{$#2$}}%
  }
  {
    #1\,\backslash\,#2
  }
  {
    #1\,\backslash\,#2
  }
  {
    #1\,\backslash\,#2
  }
}

\usepackage{fontawesome}

  \DeclareFontFamily{U}{wncy}{}
    \DeclareFontShape{U}{wncy}{m}{n}{<->wncyr10}{}
    \DeclareSymbolFont{mcy}{U}{wncy}{m}{n}
    \DeclareMathSymbol{\Sha}{\mathord}{mcy}{"58}
\begin{document}

\title{The Lyapunov spectrum for Schneider map on $p\Zp$}

\date{\today}



\author{ Matias Alvarado and Nicol\'as Ar\'evalo-Hurtado}	
\address{Instituto de Matem\'aticas, Universidad de Talca, Talca, Chile.}
\email[M. Alvarado]{matias.alvarado@utalca.cl}
\address{Universidad Escuela Colombiana de Ingenier\'ia Julio Garavito, Bogot\'a, Colombia.}

\email[N. Ar\'evalo]{nicolas.arevalo-h@escuelaing.edu.co}
\urladdr{\url{https://sites.google.com/view/nicolasarevalomath/}}

\begin{abstract}
We study the thermodynamic formalism associated with the Schneider map on the p-adic integers 
$p\Z_p$	. By introducing a geometric potential that captures the expansion of cylinder sets generated by the map, we define a Lyapunov exponent adapted to this non-Archimedean setting. We investigate the corresponding Lyapunov spectrum and show that it is real analytic on its natural domain. Moreover, we obtain an explicit closed formula for the spectrum. As a consequence, we recover and refine known results on the Hausdorff dimension of sets defined by a prescribed asymptotic arithmetic mean of the continued fraction digits. Finally, we relate the Lyapunov exponent to the exponential rate of convergence of rational approximations arising from truncations of the Schneider continued fraction expansion. This provides a 
p-adic analogue of classical results from Diophantine approximation and yielding precise dimension formulas for the associated level sets.

\end{abstract}

\maketitle

\section{Introduction}
A central problem in multifractal analysis concerns the fractal decomposition of level sets of dynamically defined functions, typically quantified via their Hausdorff dimension. Let $f:M\to M$ be a $C^{1}$ map on a compact manifold $M$. The Lyapunov exponent of $f$ at a point $x\in M$ is defined by 
\begin{align}\label{Lyapunov exponent definition}
    \lambda(x)=\lim_{n\rightarrow{\infty}}\dfrac{1}{n}\log \|Df^{n}(x)\|,
\end{align} whenever this limit exists. This quantity measures the exponential rate at which nearby orbits diverge.  For each real number $\alpha$, we define

$$J_{\alpha}=\{x\in M: \lambda(x)=\alpha\}.$$ 
By the chain rule, $\lambda(x)$ can be expressed as the pointwise limit of Birkhoff averages of $\log \|Df\|$. Consequently, for any ergodic measure, the Lyapunov exponent is constant almost everywhere. In particular, a single ergodic measure cannot capture the geometric properties of different level sets $J_{\alpha_1}$ and $J_{\alpha_2}$. This motivates the study of these sets through their Hausdorff dimension. The Lyapunov spectrum of $\alpha$ is defined by
\begin{align*}
    L(\alpha)=\dim_{\mathrm{H}}\left(J_{\alpha}\right),
\end{align*} whenever $J_{\alpha}\neq \emptyset$, where $\dim_{\mathrm{H}}(\cdot)$ denotes the Hausdorff dimension. 

For conformal expanding maps on smooth manifolds with H\"older continuous derivative, this problem dates back to the work of Weiss  \cite{we}. Using tools such as Markov partitions and Gibbs measures, Weiss proved that $L$ is defined on a bounded interval $\left[\alpha_{min},\alpha_{max}\right]$, is real analytic, and satisfies
\begin{align}\label{EcuWeiss}
    L(\alpha)=\dfrac{1}{\alpha}\inf\left\{P\left(-t\log\|Df\|\right)+t\alpha:t\in \R\right\},
\end{align} 
where $P(-t\log\|Df\|)$ denotes the topological pressure associated with the potential $-t\log\|Df\|$.
For interval maps, Pollicott and Weiss \cite{mh}, and later  Kesseb\"ohmer and Stratmann \cite{kessestrat}, extended this result to the Gauss map $G$. This map has been extensively studied due to its intimate connection with the continued fraction expansion of real numbers.

The Gauss map $G: (0,1]\rightarrow [0,1]$ is defined by $$G(x)=\dfrac{1}{x}-\left\lfloor \dfrac{1}{x}\right\rfloor,$$ where $\lfloor \cdot \rfloor$ denotes the integer part.
Every irrational $x\in (0,1)$ admits a unique continued fraction expansion  of the form
  \begin{align}\label{continued}
      x= \cfrac{1}{a_{1}+\cfrac{1}{a_{2}+\cfrac{1}{a_{3}+\cdots }}}=[a_{1},a_{2},a_{3},...],
  \end{align} where each $a_{i}$ is a positive integer and $a_{i}=\left\lfloor \frac{1}{G^{i-1}x}\right\rfloor$. 
The $n$-th rational approximation of $x$ is given by
  \begin{align*}
      \dfrac{P_n}{Q_{n}}=\cfrac{1}{a_{1}+\cfrac{1}{a_{2}+\cfrac{1}{a_{3}+\dfrac{1}{\ddots +\frac{1}{a_{n}}} }}}=[a_{1},a_{2},...,a_{n}].
  \end{align*} 
The definition of the Lyapunov exponent extends naturally to the Gauss map. For $x\in (0,1]$ we define $\lambda(x)=\lim_{n\to \infty}\frac{1}{n}\log |(G^{n})'(x)|$, whenever this limit exists. Moreover, $\lambda(x)$ can be written as  (see \cite[Page 160]{mh})
  \begin{align}\label{EcuLyaAsRationalApprox}
      \lambda(x)=-\lim_{n\rightarrow \infty}\dfrac{1}{n}\log \left|x-\dfrac{P_{n}}{Q_{n}}\right|.
  \end{align} 
By equation \eqref{EcuLyaAsRationalApprox}, $\lambda(x)$ quantifies the exponential rate at which the rational approximations $P_n/Q_n$ converge to 
$x$, linking Lyapunov exponents to classical Diophantine approximation theory. For Lebesgue-almost every $x\in (0,1)$, one has $\lambda(x)=\frac{\pi^{2}}{6\log 2}$ since $\lambda
(x)$ is a pointwise limit of Birkhoff averages. The domain of the Lyapunov spectrum $L$ for the Gauss map is $\left(2\log\left(\frac{\sqrt{5}-1}{2}\right),\infty\right)$. since $(0,1)\setminus \mathbb{Q}$ is a non-compact space, this leads to a different situation from the one considered by Weiss. Nevertheless, $L$ is a real analytic function, and for each $\alpha$ in this domain, there exists a unique ergodic measure supported on $J_{\alpha}$.

In this paper, we extend the study of Lyapunov spectra and their Diophantine implications to the Schneider map on $p\Zp$. 

Let $p$ be a prime number and let $v_p$ denote the $p$-adic valuation on $\Q$, defined by the property that for each $x\in \Q$, there exists a unique integer
$v_p(x)$ such that $x=p^{v_p(x)}\frac{m}{n}$ with $(p,mn)=1$. The $p$-adic absolute value on $\Q$ is defined by $|0|_p=0$, and $|x|_p=p^{-v_p(x)}$ if $x\neq 0$. The field $\Q_p$ of $p$-adic numbers is the completion of $\Q$ with respect to this norm. The set of all elements $x\in \Q_p$ such that $|x|_p\leq 1$ is denoted by $\Z_p$, and the subset $p\Z_p$ consists of those with $|x|_p<1$. Explicitly, any element in $\Q_p$ can be expressed as $\sum_{n\geq n_0}c_np^{n}$, where $n_0$ is an integer depending only on $x$ and $c_n\in \{0,1,...,p-1\}$. An element $x$ in $\Z_p$ can be written as $\sum_{n\geq0} c_np^n$. Thus, $p\Zp$ consists of all elements in $\Z_p$ with $c_0=0.$ 
Since $\Q_p$ is locally compact, there exists a Haar measure $\mu_p$ normalized by $\mu_{p}(p\Z_p)=1$. A detailed exposition on Haar measures can be found in \cite{folland}.

The Schneider map $T_{p}:p\Z_{p}\rightarrow p\Z_{p}$ is defined by $T_{p}(0)=0$, and for $x\neq 0$
  \begin{align*}
      T_{p}(x)=\dfrac{p^{a_{1}(x)}}{x}-b_{1}(x),
  \end{align*} where $a_{1}(x)=v_{p}(x)$ and $b_1(x)\in\{1,2,...,p-1\}$ is uniquely determined by the congruence $b_{1}(x)\equiv p^{a_{1}(x)}/x$ (mod $p$). Outside a countable set, every element of $p\Z_p$ admits a continued fraction expansion analogous to \eqref{continued}. More precisely, for $x\in p\Z_{p}\setminus \bigcup_{k\in \N} T_p^{-k}(0)$ and for each $n\in \N$, one can write
  \begin{align}\label{SchneiderContinued}
      x= \cfrac{p^{a_{1}(x)}}{b_{1}(x)+\cfrac{p^{a_{2}(x)}}{b_{2}(x)+\cfrac{p^{a_{3}(x)}}{\ddots + \cfrac{p^{a_{n}(x)}}{b_{n}(x)+T_{p}^n(x)}}}}.
  \end{align}
where $a_{i}(x)=a_{1}(T^{i-1}_{p}x)$ and $b_{i}(x)=b_{1}(T^{i-1}_{p}x)$.  Note that a point $x\in p\Zp$ has a finite continued fraction expansion if and only if $T^{n}_{p}(x)=0$ for some $n\in \N$. We denote by $F$ the set of points in $p\Z_p$ with finite continued fraction expansion. By the preceding discussion, $F=\bigcup_{k=0}^\infty T_p^{-k}(0)$. This situation is analogous to the classical Gauss map, where the points with finite continued fraction expansion are $\bigcup^{\infty}_{n=0}G^{-n}(0)=\Q$. In contrast, if $x$ is a rational number with a infinite Schneider continued fraction expansion, then there exists $N\in \N$ such that $a_{n}(x)=1$ and $b_{n}(x)=p-1$ for all $n\geq N$ (see  \cite[Theorem 1]{hw}). 
 The Schneider map is ergodic with respect to the Haar measure $\mu_p$  on $p\Z_p$ (see \cite[Lemma 2]{hw}) and its measure theoretic entropy is $\frac{p}{p-1}\log p $ (see \cite{hn}).

  From a fractal geometric perspective, Hu, Yu and Zhao (\cite{hyz}) studied multifractal decomposition determined by the asymptotic arithmetic mean of the digits appearing in the continued fraction expansion. More recently, Song, Wu, Yu and Zeng \cite{swyz} investigated the Hausdorff dimension of related level sets. We continue this line of research by studying the fractal decomposition induced by rational approximations associated with the Schneider map.

For each pair of integers $a\geq 1$ and $1\leq b\leq p-1$, we define the cylinder set 
  $$I_{(a,b)}:=\{x\in p\Z_p:a_{1}(x)=a \text{ and $b_{1}(x)=b$}\},$$
  and introduce the potential function
\begin{align*}
    \psi(x)=\dfrac{1}{\left|I_{(a_{1}(x),b_{1}(x))}\right|_{p}}=p^{a_{1}(x)}.
\end{align*} 
As will be shown, this potential plays a role analogous to $\log\|Df\|$ in the classical setting, in the sense that $\log \psi(x)$ measures the local expansion rate of $T_p$. 

For $x\in p\Z_p\setminus F$, we define the Lyapunov exponent by

\begin{align*}
     \lambda_{p}(x)=\lim_{n\rightarrow{\infty}}\dfrac{1}{n}S_{n}\log\psi (x),
\end{align*} 
whenever this limit exists, where $S_n\log\psi=\sum^{n-1}_{k=0}\log\psi\circ T^{k}_{p}$ denotes the $n$-th Birkhoff sum of $\log\psi$ with respect to $T_{p}$.  It follows that
\begin{align}
    \lambda_{p}(x)=\lim_{n\rightarrow{\infty}}\log p \cdot \dfrac{a_{1}(x)+a_{2}(x)+\cdots +a_{n}(x)}{n},
\end{align} 
so that, up to a multiplicative constant, the Lyapunov exponent coincides with the asymptotic arithmetic mean of the digits in the continued fraction expansion. The associated level sets are $$J_{p}(\alpha)=\{x\in p\Z_{p}\setminus F: \lambda_{p}(x)=\alpha\},$$ and the Lyapunov spectrum is defined by $$L_{p}(\alpha)=\mathrm{dim}_{\mathrm{H}}\,J_{p}(\alpha),$$ where $\mathrm{dim}_{\mathrm{H}}$ now denotes the Hausdorff dimension with respect to the $p$-adic norm. In the case of the Gauss map, the Lyapunov spectrum is computed via the topological pressure of the geometric potential $-t\log|G'|$, which allows one to estimate the diameters of cylinder sets using the Mean Value Theorem. However, since the Gauss map acts on a non-compact space, the thermodynamic formalism in this setting requires additional assumptions, such as the Big Images Property and H\"older continuity of the potential. Our main difficulty in developing a thermodynamic formalism for $T_p$ stems from two facts. First, the set of points in $p\Z_p$ with infinite continued fraction expansions is non-compact. Second, continuous functions on $p\Z_p$ do not, in general, satisfy a Mean Value Theorem. These considerations motivate our definition of the Lyapunov exponent $\lambda_p$, which directly captures the diameter of cylinder sets and links it to Hausdorff dimension.

Our main theorems extends to the $p$-adic setting the results of Weiss \cite{we} and Pollicott-Weiss \cite{mh} (see also \cite{io,kms,an}).

\begin{theorem}\label{PrincipalSinCalculoExplicito}
      The Lyapunov spectrum $L_{p}$ is real analytic on $[\log(p),\infty)$. For each $\alpha\geq  \log p$
    \begin{align}\label{EcuSInCalculoExplicito}
        L_{p}(\alpha)=\dfrac{1}{\alpha}\inf\left\{P(-t\log \psi)+t\alpha:t>0\right\},
    \end{align}where $P(-t\log\psi)$ is the topological pressure of $-t\log\psi$ with respect to $T_p$. The infimum is attained at a unique $t_{\alpha}>0$ such that $\frac{d}{dt}P(-t\log \psi)|_{t=t_{\alpha}}=-\alpha$. Moreover, there exists a unique equilibrium state $\mu_{t_{\alpha}}$ for $-t_{\alpha}\log \psi$ such that $\mu_{t_{\alpha}}(J_{p}(\alpha))=1$.
\end{theorem}

\begin{theorem}\label{PrincipalTheorem}
 For each $\alpha\geq \log p,$ 
    \begin{align}\label{Ecu2PrincipalTheorem}
        L_p(\alpha)=\dfrac{\log(p-1)+\log(\alpha-\log p)-\log\log p+\alpha\log_{p} \alpha-\alpha\log_{p}(\alpha-\log p)}{\alpha}.
    \end{align}
\end{theorem}

This explicit formula for the Lyapunov spectrum is remarkable. Such closed expressions are seldom available outside the context of affine iterated function systems on $\mathbb{R}^n$.

We conclude by highlighting some arithmetic consequences. 
Setting $\hat{\alpha}=\alpha/\log p$ we recover and refine previous results concerning digit means (see \cite[Theorem 2.2]{hyz}) : the set of points $x$ such that the digits $\{a_{1}(x),a_{2}(x),...\}$ have asymptotic mean $\hat{\alpha}$ has Hausdorff dimension 

   $$\dfrac{\widehat{\alpha}\log\widehat{\alpha}-(\widehat{\alpha}-1)\log(\widehat{\alpha}-1)+\log(p-1)}{\widehat{\alpha} \log p}.$$

We also establish a precise relationship between the Lyapunov exponent and the rate of convergence of the rational approximations arising from Schneider continued fraction, obtaining a $p$-adic analogue of classical Diophantine results for the Gauss map. This is
$$\lambda_p(x)=-\lim_{n\to \infty}\dfrac{1}{n}\log \left|x-\dfrac{p_n(x)}{q_n(x)}\right|_p,$$
where 
\begin{align*}
    \dfrac{p_{n}(x)}{q_{n}(x)}= \cfrac{p^{a_{1}(x)}}{b_{1}(x)+\cfrac{p^{a_{2}(x)}}{b_{2}(x)+\cfrac{p^{a_{3}(x)}}{b_{3}(x)+\cdots \dfrac{p^{a_{n}(x)}}{b_{n}(x)}}}}.
\end{align*}As a consequence, we obtain the following corollary.

\begin{corollary}\label{ApproximationSpectrumTHM}
      For each $\alpha\geq \log p$, the set 
      \begin{align*}
        \left\{x\in p\Z_{p}\setminus F:-\lim_{n\rightarrow \infty}\dfrac{1}{n}\log \left|x-\dfrac{p_{n}(x)}{q_{n}(x)}\right|_{p}=\alpha\right\}  
      \end{align*}
     has Hausdorff dimension
      \begin{align*}
          \dfrac{\log(p-1)+\log(\alpha-\log p)-\log\log p+\alpha\log_{p}\alpha-\log_{p}(\alpha-\log p)}{\alpha}.
      \end{align*}
  \end{corollary}

The proof of our main theorems combine several key ingredients.
First, we employ a compact approximation scheme by analyzing truncated subsystems where points have bounded digits.
Second, we apply the thermodynamic formalism to establish the existence and uniqueness of equilibrium states for the relevant family of geometric potentials.
Finally, we use Gibbs measures and multifractal analysis to derive dimension formulas via pointwise dimensions and Legendre transforms, and then extend the results to the full system.

\subsection*{Overview of the article} In Section 2, we collect preliminaries from symbolic dynamics, the Schneider map, and thermodynamic formalism.
In section 3, we study compact approximations of $(T_{p},p\Zp)$ and derive the multifractal spectra for these subsystems. Section 4 is devoted to the thermodynamic formalism of the Schneider map, where we establish the main results of the article. In Section 5, we relate the geometric potential to rational approximations and prove Corollary \ref{ApproximationSpectrumTHM}. Finally, in order to keep the presentation clear and focused, the technical lemmas are deferred to section 6.

\section{Preliminaries}
We briefly recall the definition of the Hausdorff dimension, for further details, see \cite[Section 2.2]{fa}. Let $Q$ be a subset of  $p\Zp$ and Let $s\geq 0$ and $\delta>0$. We define 
\begin{align*}
    \mathcal{H}^{s}_{\delta}(Q):=\inf\left\{\sum_{i}|U_{i}|_{p}^{s}:\begin{array}{c}
         \text{ $\{U_{i}\}$ is a countable cover of $Q$ such that for every $i$}  \\
         \text{ $|U_{i}|_{p}<\delta$.}
    \end{array} \right\}, 
\end{align*} where $|U_{i}|_{p}$ denotes the diameter of $U_{i}$ with respect to the $p$-adic metric. We then define $\mathcal{H}^{s}(Q)=\lim_{\delta\rightarrow 0}\mathcal{H}^{s}_{\delta}(Q)$ and the Hausdorff dimension of $Q$ by
\begin{align}\label{HAUSDORFF}
   \mathrm{dim}_{\mathrm{H}}(Q):= \inf\{s:\mathcal{H}^{s}(Q)=0\}.
\end{align} Now we define the symbolic framework of associated with the map $T_{p}$. Let $E=\N \times \{1,...,p-1\}$ and $\Sigma=E^\N$. We consider  $(\Sigma,\sigma)$ the full-shift on $\Sigma$, where $\sigma\left((a_{i},b_{i})_{i=1}^{\infty}\right)=(a_{i+1},b_{i+1})_{i=1}^{\infty}$. This is the standard full-shift on countable symbols. The topology on $\Sigma$ is generated by the cylinder sets of the form $[(a,b)]=\{(a_{i},b_{i})_{i\in \N}:a_{1}=a \text{ and $b_{1}=b$}\}$, where $(a,b)\in E$. For a finite sequence $\{(a_0,b_0),...,(a_1,b_1)\}$, we define the cylinder
\begin{align*}
    [(a_{1},b_{1}),...,(a_{n},b_{n})]=\bigcap^{n-1}_{k=0}\sigma^{-k}[(a_{k+1},b_{k+1})].
\end{align*} For each $(a,b)\in E$, let $I_{(a,b)}$ denote the set of points $x\in p\Z_{p}$ such that $a_{1}(x)=a$ and $b_{1}(x)=b$. Similarly,
\begin{align*}
    I_{(a_{i},b_{i})^{n}_{i=1}}=\bigcap_{k=0}^{n-1}T_{p}^{-k}\left(I_{(a_{k+1},b_{k+1})}\right).
\end{align*}

Observe that for any $x,y\in I_{(a,b)}$, we have 
\begin{align}\label{Equation1}
    \left|T_{p}(x)-T_{p}(y)\right|_{p}=\left|p^{a}\left(\dfrac{1}{x}-\dfrac{1}{y}\right)\right|_{p}=p^{-a}\dfrac{|y-x|_{p}}{|xy|_{p}}.
\end{align}
Since $|x|_{p}=|y|_{p}=p^{-a}$, it follows that $\left|T_{p}(x)-T_{p}(y)\right|_{p}=p^{a}|x-y|_{p}$. Therefore, the Schneider map is locally expanding, with expansion factor $p^{a}$ on each cylinder $I_{(a,b)}$. Consequently, the countable partition $\{I_{(a,b)}\}_{(a,b)\in E}$ generates the topology on $p\Z_p\setminus F$. Define the map $\pi:\Sigma\to p\Z_p\setminus F$ given by
\begin{align*}
    (a_{i},b_{i})_{i\in \N}\mapsto \bigcap_{i=0}^{\infty}T_{p}^{-i}\left(I_{(a_{i+1},b_{i+1})}\right).
\end{align*}
Since {$\pi\left([a,b]\right)=I_{(a,b)}$} and each $I_{(a,b)}$ is both open and closed in the $p$-adic topology, $\pi$ defines a homeomorphism between $\Sigma$ and $p\Zp\setminus F$. In particular, $(\Sigma,\sigma)$ and $(p\Zp\setminus F,T_{p})$ are topologically conjugated (\cite{walters2000introduction}). 
Henceforth, for real-valued functions $f\colon p\Z_p\to \R$, and $g\colon \Sigma\to \R$, we denote by $S_n f(x)$ and $S_n g(x)$ the Birkhoff sums $\sum_{k=0}^{n-1} f\circ T_p^k(x)$ and $\sum_{k=0}^{n-1} g\circ \sigma^k(x)$ respectively.

A function $\rho:p\Z_{p}\setminus F\rightarrow \R$ is said to be \textit{H\"older} continuous with exponent $0<\alpha \leq 1$ if there exists a constant $C>0$ such that for all $x,y\in p\Zp$,
\begin{align*}
    |\rho(x)-\rho(y)|\leq C|x-y|^\alpha_{p}.
\end{align*}
For any $n\geq 1$, we define the variation of $\rho$ on cylinders of length $n$ by
\begin{align*}
 \mathrm{Var}_{n}(\rho):=\sup\{|\rho(x)-\rho(y)|:a_{i}(x)=a_{i}(y) \text{ and $b_{i}(x)=b_{i}(y)$ for $0\leq i\leq n-1$}\}.
\end{align*} 
If $\rho:p\Zp \setminus F\to \mathbb{R}$ is H\"older continuous, 
the \emph{topological pressure of $\rho$} is defined by  

\begin{align}\label{DefPressure}
    P(\rho) := \sup \left\{ h_{\mu} + \int \rho \, d\mu : \mu \in \mathcal{M}(p\Zp\setminus F,T_{p}), -\int \rho \, d\mu < \infty \right\},
\end{align}where \( h_{\mu} \) denotes the measure-theoretic entropy of \( \mu \), and \( \mathcal{M}(p\Zp\setminus F,T_{p}) \) denotes the set of \( T_{p} \)-invariant probability measures. Any measure achieving the supremum is called an \textit{equilibrium state} for \( \rho\). If $\sum_{k\geq 1}\mathrm{Var}_{k}(\rho)<\infty$, then $P(\rho)$ can be computed using periodic points (\cite[Cor. 1]{bs})
\begin{align}\label{EcuPressure}
    P(\rho)=\lim_{n\rightarrow \infty}\dfrac{1}{n}\log\left(\sum_{T_{p}^{n}x=x}e^{-S_{n} \rho(x)}\right).
\end{align}

\section{Restricting the Schneider map to bounded valuations}

\noindent Let us restrict the Schneider map to the set of points $x\in p\Zp$ for which the digits $a_i(x)$ are bounded by a fixed integer $n\geq 1$. 
We denote this subset by $p\Z_{p,n}$, and write $T_{p,n}$ for the restriction of $T_{p}$ to $p\Z_{p,n}$. Let $E_{n}=\{1,...,n\}\times \{1,...,p-1\}$ and $\Sigma_{n}=E_{n}^{\mathbb{N}}$. Analogously to $(T_{p},p\Zp)$, the subsystem $(T_{p,n},p\Z_{p,n})$ is topologically conjugated to the full-shift $(\Sigma_{n},\sigma)$ via the map
\begin{align*}
    \pi_{n}:\Sigma_{n}&\longrightarrow p\Z_{p,n}\setminus F\\
    (a_{i},b_{i})_{i\in \N}&\longmapsto \bigcap_{i=0}^{\infty}T_{p,n}^{-i}\left(I_{(a_{i+1},b_{i+1})}\right).
\end{align*}
For simplicity, we denote by $P_{n}$ the topological pressure, whether computed on $\Sigma_{n}$ or on $p\Z_{p,n}\setminus F$. For H\"older continuous potentials, $P_{n}$ can be computed using periodic points as in \eqref{EcuPressure}. Our proof of Theorem \ref{PrincipalSinCalculoExplicito} follows the spirit of \cite[Theorem 4.1]{io} and \cite[Theorem 5.1]{an}. To that end, we first establish an analogous result for the Lyapunov spectrum of $T_{p,n}$. In fact, we obtain a complete multifractal description of these dynamical systems, summarized in Theorem \ref{PrincipalPesinWeiss} below. Let $\nu$ be any $T_{p,n}$-invariant measure. For $x\in p\Z_{p,n}$, the \textit{upper} and \textit{lower pointwise dimensions of $\nu$} at $x$ are defined by
    \begin{align*}
        \overline{d}_{\nu}(x)=\limsup_{r\rightarrow 0}\dfrac{\log \nu(B(x,r))}{\log r}, \text{ and }\underline{d}_{\nu}(x)=\liminf_{r\rightarrow 0}\dfrac{\log \nu(B(x,r))}{\log r}.
    \end{align*}If these two limits coincide, their common value is denoted by $d_{\nu}(x)$ and is called the \textit{pointwise dimension of $\nu$ at $x$}. It quantifies how the measure $\nu$ distributes upon different scales (see \cite[Chapter 17]{fa} and \cite{yh2}). The measure $\nu$ is said to be exact dimensional if $d_{\nu}(x)$ exists and is constant $\nu$-almost everywhere on $p\Z_{p,n}$. Exact dimensional measures describe the fine-scale geometry (see \cite{yh2}) through the fractal decomposition given by the level sets 
\begin{align*}
    K_{\alpha}=\{x\in p\Z_{p,n}: d_{\nu}(x)=\alpha\}.
\end{align*} The associated dimension spectrum is defined by
\begin{align}
    f_{\nu}(\alpha)=\dim_{\mathrm{H}}(K_{\alpha}).
\end{align} Let $K_{\infty}$ denote the set of points where the pointwise dimension does not exist. Then
\begin{align*}
    p\Z_{p,n}=K_{\infty}\cup \bigcup_{\alpha \in \R}K_{\alpha}.
\end{align*}

Given a H\"older continuous function $\zeta$ on $(\Sigma_{n},\sigma)$. A $\sigma$-invariant probability measure $\mu$ on $\Sigma_{n}$ is called a \textit{Gibbs measure} for $\zeta$ if there exists a  constant $C\geq 1$ such that, for all $m\in \N$, $w\in E_{n}^{m}$, and all $(x_{i})_{i\in \N}\in [w]$
    \begin{align*}
        C^{-1}\leq \dfrac{\mu\left([w]\right)}{\exp\left(S_{m}\zeta(x_{i})-mP_{n}(\zeta)\right)}\leq C.
    \end{align*} Let $\mu_{\zeta}$ be the unique equilibrium state for $\zeta$ (see \cite{pu} for the existence and uniqueness of $\mu_{\zeta}$). 

Now let $\rho:p\Zp \setminus F\to \mathbb{R}$ be H\"older continuous, and define $\zeta= \rho \circ \pi_{n}$. Set $\nu_{\rho}=  (\pi_{n})_{*}\mu_{\zeta}$ the pushforward measure of $\mu_{\zeta}$. Then $\zeta$ is H\"older continuous on $\Sigma_{n}$, $\mu_{\zeta}$ is a Gibbs measure for $\zeta$, and consequently $\nu_{\rho}$ is an equilibrium state for $\rho$. Define the H\"older continuous function $\vartheta$ on $\Sigma_{n}$ by $\log\vartheta=\zeta-P_{n}(\zeta)$. By construction, $P_{n}(\log\vartheta)=0$, and $\mu_{\zeta}$ is also the equilibrium state for $\log\vartheta$. Moreover, the map $q\mapsto P_{n}(q\log\vartheta)$ is differentiable and strictly convex for all $q\in \R$ \cite[Section 5.6]{pu}. It is a classical result in thermodynamic formalism that for each $q\in \R$ there exists a unique $\mathcal{T}(q)\in \R$ such that the potential
\begin{align*}
    \phi_{q,n}=\mathcal{T}(q)\log(\psi\circ\pi_{n})+q\log\vartheta ,
\end{align*} satisfies $P_{n}(\phi_{q,n})=0$. Each potential in the one-parameter family $\{\phi_{q,n}\}_{q\in \R}$ is H\"older continuous and admits a unique Gibbs measure $\mu_{q}$. 

Theorem \ref{PrincipalPesinWeiss} shows that the potential function $\log\psi$ plays the same role as $\log \|Df\|$ in the theory of conformal expanding maps on compact manifolds (\cite[Theorem 1]{yh}).Although the proof follows the same general strategy, several technical adaptations are required. First, $\log\psi(x)=a_{1}(x)\log p$ controls the diameter of cylinder sets, and for all $x\in p\Z_{p,n}\setminus F$ the product $\prod_{k=0}^{n}p^{-a_{1+k}(x)}$ decays exponentially. Second, the following Bowen-type formula provides the Haussdorf dimension of the attractor. See the last section for a complete proof of Lemma \ref{BowenFormula} and Theorem \ref{PrincipalPesinWeiss}.
\begin{lemma}\label{BowenFormula}
    For each $n\in \N$ we have $\mathcal{T}(0)=\dim_{\mathrm{H}}\left(p\Z_{p,n}\setminus F\right)$, that is,
    \begin{align*}
        P_{n}
        \left(-\dim_{\mathrm{H}}\left(p\Z_{p,n}\setminus F\right)\log \psi\right)=0.
    \end{align*} Moreover, $\dim_{\mathrm{H}}\left(p\Z_ {p,n}\setminus F\right)$ is the unique positive real number $\tilde{s}$ satisfying $$\sum^{n}_{k=1}\dfrac{1}{p^{k\tilde{s}}}=\dfrac{1}{p-1}.$$
\end{lemma}

\begin{theorem}\label{PrincipalPesinWeiss} 
For each $n\in \N$, the following holds:
\begin{enumerate}
    \item The measure $\nu_{\rho}$ is exact dimensional. Moreover, for $\nu_{\rho}$-almost every $x\in p\Z_{p,n}\setminus F$
        $$
             d_{\nu_{\rho}}(x)=\displaystyle \cfrac{\int_{\Sigma_{n}}\log\vartheta \ d\mu_{\zeta}}{-\int_{\Sigma_{n}}\log(\psi\circ \pi_{n}) \ d\mu_{\zeta}}.
        $$
    \item The function $q\mapsto \mathcal{T}(q)$ is real analytic, $\mathcal{T}(1)=0$ and $\mathcal{T}''(0)\geq 0$.
    \item The function $\alpha(q)
    :=-\mathcal{T}'(q)$ takes values in $[\log p,n\log p]$ and $\nu_{\rho}$ is fully supported in $p\Z_{p,n}\setminus F$. Furthermore, 
        \begin{align*}
            f_{\nu_{\rho}}(\alpha(q))=\mathcal{T}(q)+q\alpha(q).
        \end{align*}
    \item  For any $q\in \R$ 
     \begin{align*}
            \mathcal{T}(q)=-\lim_{r\to 0}\dfrac{\log\left(\inf_{\mathcal{V}}\sum_{B\in \mathcal{V}}\nu_{\rho}(B)^{q}\right)}{\log r},
        \end{align*}where the infimum is taken over all finite covers $\mathcal{V}$ of $p\Z_{p,n}\setminus F$ by open balls $B$ of radius $r$.
\end{enumerate}
\end{theorem}
 
We now restate a result of Weiss connecting this formalism with the Lyapunov exponent, which in our setting follows from the same proof (see \cite[Theorem 2.3]{we}).
\begin{theorem}\label{Theorem 2.3 Weiss}
    For each $n\in \N$ and for $\nu_{\rho}$-almost every $x\in p\Z_{p,n}\setminus F$ we have
    \begin{align*}
        d_{\nu_{\rho}}(x)=\dfrac{P_{n}(\rho)-\overline{\rho}(x)}{\lambda_{p}(x)}=\dfrac{h_{\nu_{\rho}}+\int \rho d\nu_{\rho}-\overline{\rho}(x)}{\lambda_{p}(x)},
    \end{align*}where $\displaystyle \overline{\rho}(x)=\lim_{m\to \infty}\frac{1}{m}\sum^{m-1}_{k=0}\rho\left(T_{p,n}^{k}(x)\right)$.
\end{theorem}
\noindent Since both $\overline{\rho}$ and $\lambda_p$ are constant $\nu_\rho$-almost everywhere, this yields the desired relation, analogous to equation \eqref{EcuWeiss}.
Let $L_{p,n}$ denote the Lyapunov spectrum associated with $\left(T_{p,n},p\Z_{p,n}\setminus F\right)$, that is, $L_{p,n}(\alpha)=\dim_{\mathrm{H}}(J_{p,n}(\alpha))$ where $J_{p,n}(\alpha)=\{x\in p\Z_{p,n}\setminus F:\lambda_{p}=\alpha\}$.

\begin{theorem}\label{LyapunovSpectrumCompactAproximation}
    For each $\alpha \in[\log p,n\log p]$,
    \begin{align}\label{ECCU}
        L_{p,n}(\alpha)=\dfrac{1}{\alpha}\inf_{t\in \R}\{P_{n}(-t\log \psi )+t\alpha\}=\dfrac{1}{\alpha}P_{n}(-t_{\alpha}\log\psi)+t_{\alpha},
    \end{align}where $t_{\alpha}$ is the unique real number $t$ satisfying 
    \begin{align}\label{ECCU2}
        -\alpha=\log p\cdot  \left(\dfrac{n}{p^{tn}-1}-\dfrac{p^{t}}{p^{t}-1}\right).
    \end{align}
    Moreover, there exists a unique equilibrium state $\nu_{t_{\alpha}}$ of the potential $-t_{\alpha}\log\psi$ which is supported on  $\{x\in p\Z_{p,n}\setminus F:\lambda_{p}(x)=\alpha\}$ .
\end{theorem}

\begin{proof}
The middle term in equation \eqref{ECCU} is known to be the Legendre transform of $P_{n}$. For each $t\in \R$, consider the H\"older continuous potential $-t\log\psi(x)$, and let $\nu_{t,n}$ denote its unique equilibrium state. By Theorem \ref{Theorem 2.3 Weiss}, for $\nu_{t,n}$-almost every $x\in p\Z_{p,n}\setminus F$, we have
    \begin{align}\label{EcuCorollaryWeiss}
        \lambda_{p}(x)=\dfrac{P
        _{n}(-t\log \psi)}{d_{\nu_{t,n}}(x)-t},
    \end{align} where $\overline{\rho}(x)=\overline{-t\log\psi}(x)=t\lambda_{p}(x)$. Since $\nu_{t,n}$ is exact dimensional, it follows that $\lambda_{p}(x)$ is constant $\nu_{t,n}$-almost everywhere. Moreover, $\lambda_{p}(x)=-\frac{d}{dt}P_{n}(-t\log\psi)$ (see \cite[Chapter 5]{pu}). Therefore, since the function $t\mapsto P_{n}\left(-t\log\psi\right)$ is strictly convex, its derivative is injective, and the function $L_{p,n}(\alpha)$ is given by the Legendre transform of $P_{n}(-t\log \psi)$ (see \cite{yh}). This yields the infimum in equation \eqref{ECCU}. As the equilibrium states of geometric potentials $-t\log \psi$ support the level sets of the Lyapunov exponent, the values $\alpha$ in the domain of $\lambda_p$ correspond precisely to the range of $-\frac{d}{dt}P_{n}(-t\log\psi)$. Since $P_{n}(-t\log\psi)=\log(p-1)+\log\left(\frac{p^{tn}-1}{p^{tn}(p^{t}-1)}\right)$, we obtain
\begin{align*}
    \dfrac{d}{dt}P_{n}(-t\log\psi)=\log p\cdot \left(\dfrac{n}{p^{tn}-1}-\dfrac{p^{t}}{p^{t}-1}\right).
\end{align*}The range of this derivative determines the domain of $L_{p,n}$, namely $[\log p,n\log p]$. By the convexity of $t\mapsto P_{n}(-t\log\psi)$, the infimum in equation \eqref{ECCU} is attained at $t_{\alpha}$, the unique real solution of $\frac{d}{dt}P_{n}(-t\log\psi)|_{t={t_{\alpha}}}=-\alpha$.
\end{proof}
\begin{remark}
    \noindent Note that $P_{n}(-t\log\psi)=\log(p-1)+\log\left(\frac{p^{tn}-1}{p^{tn}(p^{t}-1)}\right)$ implies that the topological entropy $h(T_{p,n})=\log(p-1)+\log n$, since $P(0)=h(T_{p,n})$.
\end{remark}
\begin{example}
 Let us compute $t_\alpha$ explicitly for $n=2.$
 Set $X=p^t$, and $k=\frac{-\alpha}{\log p}$. From equation \eqref{ECCU2} we must solve
 \begin{align}\label{EcuEjemploP2}
     k=\dfrac{2}{X^2-1}-\dfrac{X}{X-1},
 \end{align}
 which simplifies to $$(k+1)X^2+X-(k+2)=0.$$ Thus
 $$X=\dfrac{-1\pm\sqrt{(2k+3)^2}}{2(k+1)}=\dfrac{-1\pm |2k+3|}{2(k+1)}.$$
The only admissible solution of equation \eqref{EcuEjemploP2} is $X=-\frac{k+2}{k+1}$, then $p^{t_\alpha}=-\frac{k+2}{k+1}$. Note that when $\alpha=-\frac{3}{2}\log p$, equation \eqref{EcuEjemploP2} has no solution; but since $\frac{d}{dt}P(-t\log\psi)|_{t=0}=-\frac{3}{2}\log p$, we obtain $t_{\frac{3}{2}\log p}=0$ due to injectivity of $\frac{d}{dt}P_{2}(-t\log \psi)$. Using $P_{2}(-t\log \psi)=\log(p-1)+\log\left(\frac{p^{2t}-1}{p^{2t}(p^{t}-1)}\right)$ and substituting $k=-\frac{\alpha}{\log p }$, we obtain
 \begin{align*}
     P_{2}(-t_{\alpha}\log \psi ) &=\log(p-1)+\log\left(\dfrac{-k-1}{(k+2)^{2}}\right)\\
     &=\log(p-1)+\log \log p+\log(\alpha-\log p)-2\log(2\log p-\alpha).
 \end{align*}
 From direct computation, Theorem \ref{LyapunovSpectrumCompactAproximation} implies that
 \begin{align*}
     L_{p,2}(\alpha)&=\dfrac{\log(p-1)+\log \log p}{\alpha}+\log(\alpha-\log(p))\left(\frac{1}{\alpha}-\frac{1}{\log p}\right)\\
     &+\log(2\log p-\alpha)\left(\frac{1}{\log p}-\frac{2}{\alpha}\right),
 \end{align*}and $L_{p,2}(\frac{3}{2}\log p)=\frac{2}{3}\left(\frac{\log(p-1)+\log 2}{\log p}\right)$.
 From this, we can estimate that $L_{p,2}(\log p)=\frac{\log(p-1)}{\log p}$, and $L_{p,2}(2\log p)=\frac{\log(p-1)}{2\log p}$. The maximum value of $L_{p,2}$ (see Lemma \ref{BowenFormula}) is the only real $s$ such that $p^{2s}=(p-1)p^{s}+(p-1)$ which is 
 \begin{align*}
     \dim_{\mathrm H}(p\Z _{p,2})=\log_{p}\left(\dfrac{p-1+\sqrt{p^{2}+2p-3}}{2}\right).
 \end{align*} From Theorem \ref{LyapunovSpectrumCompactAproximation}, this value is attained at 
 \begin{align*}
     \alpha=\log p \cdot\left(\dfrac{p-1+\sqrt{D}}{p-3+\sqrt{D}}-\dfrac{2}{D+(p-1)\sqrt{D}}\right),
 \end{align*} where $D=p^{2}+2p-3$.
\end{example}
\begin{remark}
    In contrast with Theorem \ref{PrincipalTheorem}, to obtain an explicit expression of the Lyapunov spectrum $L_{p,n}$ in Theorem \ref{LyapunovSpectrumCompactAproximation} one must find the positive root of 
    \begin{align*}
        X^{n}(K+1)+X^{n-1}+\cdots +X+K-n=0,
    \end{align*} where $K=\frac{-\alpha}{\log p}$. Analogously, from Lemma \ref{BowenFormula} the Haussdorf dimension $\dim_{\mathrm H}\left(p\Z_{p,n}\setminus F\right)$ is obtained as the positive root of 
    \begin{align*}
        X^{n+1}+pX^{n}+p-1=0.
    \end{align*} Moreover, for the explicit value $\alpha=\frac{n+1}{2}\log p$ we have $t_{\alpha}=0$, since $\frac{d}{dt}P_{n}(-t\log \psi )|_{t=0}=-\frac{n+1}{2}\log p$. Consequently, 
    \begin{align*}
        L_{p,n}\left(\frac{n+1}{2}\log p\right)=\frac{2}{n+1}\left(\frac{\log(p-1)+\log n}{\log p}\right).
    \end{align*}
\end{remark}

\section{Proof of main theorems}

 We begin this section by developing the thermodynamic formalism associated with the Schneider map. Recall that $\log \psi $ is locally constant on cylinders of length 1. Hence, from equation \eqref{EcuPressure} we obtain 
\begin{align}\label{TopPressureSarig}
    P\left(-t\log \psi\right)=\lim_{n\rightarrow \infty}\dfrac{1}{n}\log\left(\sum_{(a_i,b_i)_{i=1}^{n}\in E^n} \, \prod_{k=1}^{n}(p^{a_{k}})^{-t}\right)=\log\left(\dfrac{p-1}{p^{t}-1}\right).
\end{align}
By Sarig's results on Countable Markov shifts (see \cite{sa2}) $P_{n}(-t\log \psi)\to P(-t\log \psi)$ as $n \to \infty$. Recall $\lim_{t\to 0^{+}}P(t)=\infty$, then for each $\alpha> \log p$ and $n\in \N$ sufficiently large there exists a unique $t_{\alpha,n}\in \R$ such that $-\alpha=\frac{d}{dt}P_{n}(-t_{\alpha,n}\log \psi)$. Let $\nu_{\alpha,n}$ denote the equilibrium state of $-t_{\alpha,n}\log\psi$. 

Let $\varepsilon>0$. Then
\begin{align*}
    \alpha=\int_{p\Z_{p,n}}\log \psi \  d\nu_{\alpha,n}=\int_{p\Z_{p,n}\setminus B(0,\varepsilon)}\log\psi \ d\nu_{\alpha,n}+\int_{B(0,\varepsilon)}\log\psi \ d\nu_{\alpha,n}.
\end{align*} Define $M_{\varepsilon}:=\inf\{\log\psi(x):x\in B(0,\varepsilon)\}$. Since $a_1(x)\leq n$ for $x\in p\Z_{p,n}$ we have $M_\varepsilon\leq n\log(p).$ Hence
\begin{align*}
    \alpha \geq \int_{p\Z_{p,n}\setminus B(0,\varepsilon)}\log \psi \ d\nu_{\alpha,n}+M_{\varepsilon}\cdot \nu_{\alpha,n}(B(0,\varepsilon)),
\end{align*}which implies 
\begin{align*}
    \dfrac{\alpha}{M_{\varepsilon}}\geq \nu_{\alpha,n}(B(0,\varepsilon)).
\end{align*}Therefore, for every $\delta>0$ and $n$ sufficiently large, there exists $\varepsilon(\delta)>0$ such that 
\begin{align*}
    \nu_{\alpha,n}\left(p\Z_{p,n}\setminus B(0,\varepsilon)\right)=1- \nu_{\alpha,n}(B(0,\varepsilon))\geq 1-\dfrac{\alpha}{M_{\varepsilon}}\geq 1-\dfrac{\alpha}{n\log(p)}\geq \delta.
\end{align*}In other words, $\{\nu_{\alpha,n}\}_{n\in \N}$ is a tight sequence of probability measures. By Prohorov's theorem, it admits a weak-* convergent subsequence. If $\nu_{\alpha}$ is an accumulation point, then $\nu_{\alpha}(p\Z_{p}\setminus F)=1$. Indeed, $\{t_{\alpha,n}\}_{n\in \N}$ also admits a convergent subsequence, since $\{\frac{d}{dt}P_{n}(-t\log \psi)\}_{n\in \N}$ is a decreasing sequence of functions and $\frac{d}{dt}P_{n}(-t\log \psi)\geq \frac{d}{dt}P(-t\log \psi )$ for all $n$.
If $t_{\alpha}$ is an accumulation point of $\left\{ t_{\alpha,n} \right\}$, then $\nu_{\alpha}$ is an equilibrium state for the potential $-t_{\alpha}\log \psi $ (see \cite[section 6]{io}).
Combining equation  \eqref{EcuPressure} with the thermodynamic formalism for countable Markov shifts (see \cite{sa2}), we obtain the following result.

\begin{theorem}
For all $t\in \R$
\begin{align}\label{TopPressure}
    P(-t\log \psi)=\begin{cases}
        \log\left(\dfrac{p-1}{p^{t}-1}\right) & \text{if $t>0$}\\
        \infty & \text{if $t\leq 0$},
    \end{cases}
\end{align} and for each $t>0$ there exists a unique equilibrium state $\nu_t$ of $-t\log\psi$.  If $\alpha=-\frac{d}{dt}P(-t\log \psi)|_{t=t_{\alpha}}$, then $\nu_{t_\alpha}$ is a weak-* accumulation point of the sequence $\{\nu_{\alpha,n}\}_{n\in \N}$.
\end{theorem}
\begin{remark}\label{Remark1}
In this setting, the equilibrium state for $-\log\psi$ ( when $t=1$) coincides with the Haar measure. Indeed, $P(-\log \psi)=0$ and $\mu_{p}$ attains the supremum in equation \eqref{DefPressure} since
\begin{align}
    h_{\mu_{p}}-\int \log\psi \ d\mu_{p}= \frac{p}{p-1}\log p-\frac{p}{p-1}\log p=0,
\end{align}where $h_{\mu_{p}}=\frac{p}{p-1}\log p$ (\cite[Theorem 1.1]{hn}) and $\int \log\psi \ d\mu_{p}=\frac{p}{p-1}\log p$ (\cite[Theorem 3]{hw}). Furthermore, $L_{p}$ achieves its maximum value $\dim_{\mathrm{H}}(p\Zp\setminus F)=1$ at $\alpha=\frac{p\log p}{p-1}$.
\end{remark}

\textit{Proof of Theorem \ref{PrincipalSinCalculoExplicito}.}
By Theorem \ref{LyapunovSpectrumCompactAproximation}, for each $n\in \N$ and $\alpha \in [\log p,n\log p]$ 
    \begin{align}\label{Ecu19}
        L_{p,n}(\alpha)=\dfrac{1}{\alpha}\inf\{P_{n}(-t\log \psi)+t\alpha\}.
    \end{align} Since $ p\Z_{p,n}\setminus F\subset  p\Z_{p_,n+1}\setminus F$ for every $n\geq 1$, the sequence $\{L_{p,n}(\alpha)\}_{n\in \N}$ is monotone increasing and bounded above by the Hausdorff dimension of $p\Z_p$, which equals one. From equation \eqref{Ecu19}, for each $\alpha\geq \log p$
    \begin{align}\label{Ecu20}
        \lim_{n\to \infty}\dfrac{1}{\alpha}\inf_{t\in \R}\{P_{n}(-t\log \psi)+t\alpha\}\leq L_p(\alpha),
    \end{align}since $J_{p,n}(\alpha)\subset J_{p}(\alpha)$ for all $n\in \N$. To show that the Legendre transform of $P$ is bounded above by $L_p(\alpha)$, we employ the notion of infimally convergence. We say that the sequence $\{P_{n}\}_{n\in \N}$ \textit{converges infimally} to $P$ if, for every $t\in \R$ \begin{align}\label{InfimalCon} 
    \lim_{\beta\rightarrow 0}\liminf_{n\rightarrow \infty}\inf\{P_{n}
(-s\log \psi):|t-s|<\beta\}=\lim_{\beta\rightarrow 0}\limsup_{n\rightarrow \infty}\inf\{P_{n}
(-s\log\psi):|t-s|<\beta\}=P(-t\log\psi).
\end{align} Recall that the map $t\mapsto P_{n}(-t\log \psi)$ is non-increasing. Hence, for any $\beta>0$, $$\inf\{P_{n}(-s\log \psi):|t-s|<\beta\}=P(t+\beta).$$ On the other hand, from \cite[Theorem 2]{sa2}, for all $t\in \R$, 
\begin{align}\label{SarigConvergence}
    \lim_{n\rightarrow \infty}P_{n}(-t\log \psi)=P(-t\log \psi).
\end{align} We obtain $\limsup_{n\rightarrow \infty}$ and $\liminf_{n\rightarrow \infty}$ in equation \eqref{Ecu20} coincide for every $\beta>0$. Therefore, $\{P_{n}\}_{n\in \N}$ converges infimally to $P$.

By Wijsman's work on infimally convergence (see \cite[Theorem 6.2]{wi2}), infimmally convergence of $\{P_{n}\}_{n\in \N}$ to $P$ is equivalent to infimmaly convergence of their Legendre transforms. In particular, the sequence of functions $\{\frac{1}{\alpha}\inf_{t\in \R}\{P_{n}(-t\log \psi)+t\alpha\}\}_{n\in \N}$ converge infimally to $\frac{1}{\alpha}\inf_{t\in \R}\{P(-t\log \psi)+t\alpha\}$ as functions on $\alpha$. Moreover, the function obtained by pointwise convergence dominates that arising from infimally convergence (see \cite[section 5]{wi2}). This implies that $$\dfrac{1}{\alpha}\inf_{t\in \R}\{P(-t\log \psi)+t\alpha\}\leq \lim_{n\rightarrow \infty}\dfrac{1}{\alpha}\inf_{t\in \R}\{P_{n}(-t\log \psi)+t\alpha\}.$$ Together with inequality \eqref{Ecu20} we get 
\begin{align}\label{AimFOR}
    \dfrac{1}{\alpha}\inf_{t\in \R}\{P(-t\log \psi)+t\alpha\}\leq L_{p}(\alpha).
\end{align} Now we prove the reverse inequality. Fix $t>0$ and $\varepsilon >0$. Since the function $t\mapsto P(-t\log \psi)$ is strictly decreasing on $[0,\infty)$, by definition of topological pressure we obtain
\begin{align*}
    P_{\varepsilon}:=P(-(t+\varepsilon)\log \psi)-P(-t\log\psi)=P(-(t+\varepsilon)\log\psi)-P(-t\log\psi)<0.
\end{align*} Equation \eqref{EcuPressure} implies there exists $M\in \N$ such that for all $k\geq M$
\begin{align*}
    \sum_{x:T^{k}_{p}x=x}e^{-(t+\varepsilon)S_{k}\log\psi-kP(-t\log\psi)}<e^{\frac{P_{\varepsilon}k}{2}}<1.
\end{align*}Let $\delta>0$ such that $\frac{P_{\varepsilon}}{2}-\delta\frac{P(-t\log\psi)}{\alpha}<0$. On the other hand, for every $y\in J_{p}(\alpha)$ there exists $N_{y}\in \N$ such that for all $k\geq N_{y}$
\begin{align*}
    \left|\dfrac{S_{k}\log \psi}{k}-\alpha\right|<\delta.
\end{align*}Using this fact, we construct for each $k\in \N$ a countable cover of cylinder sets with diameter at most $p^{-k}$. Notice that each interval $I_{(a_{i},b_{i})_{i=1}^{n}}$ satisfies
\begin{align*}
    N[(a_{i},b_{i})_{i=1}^{n}]:=\min\{N_{y}:y\in J_{p}(\alpha)\cap I_{(a_{i},b_{i})_{i=1}^{n}}\}<\infty.
\end{align*} We define 
\begin{align*}
    \overline{A_{k}}:=\{I_{(a_{i},b_{i})_{i=1}^{k}}:N[(a_{i},b_{i})_{i=1}^{n}]\leq k\},
\end{align*}and for every $n>k$ we set
\begin{align*}
    A_{n}:=\{I_{(a_{i},b_{i})_{i=1}^{n}}:N[(a_{i},b_{i})_{i=1}^{n}]=n\}.
\end{align*}Let $\mathcal{A}_{k}=\overline{A}_{k}\cup \bigcup_{n>k}A_{n}$. The set $\mathcal{A}_{k}$ is a countable cover of $J_{p}(\alpha)$ since each $A_{n}$ is a countable collection of cylinders. On the other hand, recall that $T_{p}$ is distance expansive, then for any cylinder $I_{(a_{i},b_{i})_{i=1}^{n}}\in \mathcal{A}_{k}$ we have that $|I_{(a_{i},b_{i})_{i=1}^{n}}|_{p}\leq p^{-n}$. 

Recall that $|I_{(a_{i},b_{i})_{i=1}^{n}}|_{p}=e^{-S_{n}\log\psi(y)}$ for all $y\in I_{(a_{i},b_{i})_{i=1}^{n}}$. Thus, following the definition of Hausdorff dimension, for $l=t+\varepsilon+\frac{P(-t\log \psi)}{\alpha}$ we get
\begin{align*}
    \mathcal{H}^{l}_{p^{-k}}(J_{p}(\alpha))
    &\leq \sum_{I_{(a_{i},b_{i})_{i=1}^{n}}\in \mathcal{A}_{k}}e^{-lS_{n}\log\psi(x_{(a_{i},b_{i})})},
\end{align*}where $x_{(a_{i},b_{i})}$ is the periodic point of $I_{(a_{i},b_{i})^{n}_{i=1}}$ with period $n$. Since $N[I_{(a_{i},b_{i})_{i=1}^{n}}]\leq n$,
\begin{align*}
    -lS_{n}\log\psi(x_{(a_{i},b_{i})})&=-\left(t+\varepsilon+\dfrac{P(-t\log\psi)}{\alpha}\right)S_{n}\log\psi(x_{(a_{i},b_{i})})\\
    & \leq -(t+\varepsilon)S_{n}\log\psi(x_{(a_{i},b_{i})}))-n\frac{P(-t\log\psi)}{\alpha}(\delta-\alpha).
\end{align*} Therefore,
\begin{align*}
    \mathcal{H}^{l}_{p^{-k}}(J_{p}(\alpha))&\leq \sum_{I_{(a_{i},b_{i})_{i=1}^{n}}\in \mathcal{A}_{k}}e^{-(t+\varepsilon)S_{n}\log\psi(x_{(a_{i},b_{i})})-n\frac{P(-t\log\psi)}{\alpha}(\delta-\alpha)}\\
    & = \sum^{\infty}_{n=k}e^{-\frac{n\delta P(-t\log\psi)}{\alpha}}\sum_{I_{(a_{i},b_{i})_{i=1}^{n}}\in \mathcal{A}_{k}}e^{-(t+\varepsilon)S_{n}\log\psi(x_{(a_{i},b_{i})}))+nP(-t\log\psi)}\\
    & = \sum^{\infty}_{n=k}e^{-\frac{n\delta P(-t\log\psi)}{\alpha}}\sum_{x:T_{p}^{n}x=x}e^{-(t+\varepsilon)S_{n}\log\psi(x)+nP(-t\log\psi)}.
\end{align*}Thus, for all $k\geq M$ we have that 
\begin{align}\label{Ecu21}
    \mathcal{H}^{l}_{p^{-k}}(J_{p}(\alpha))&\leq \sum^{\infty}_{n=k}e^{-\frac{n\delta P(-t\log\psi)}{\alpha}+\frac{P_{\varepsilon}n}{2}}=\sum^{\infty}_{n=k}e^{n(\frac{P_{\varepsilon}}{2}-\frac{\delta P(-t\log\psi )}{\alpha})}.
\end{align} Since $\dfrac{P_{\varepsilon}}{2}-\frac{\delta P(-t\log(\psi))}{\alpha}<0$, we obtain a geometric series on the right-hand side of inequality \eqref{Ecu21}. Therefore,
\begin{align*}
    \mathcal{H}^{l}(J_{p}(\alpha))=\lim_{k\to \infty}\mathcal{H}^{l}_{p^{-k}}(J_{p}(\alpha))=0.
\end{align*}This implies that $\dim_{\mathrm{H}}(J_{p}(\alpha))=L_{p}(\alpha)\leq t+\varepsilon+\frac{P(-t\log(\psi))}{\alpha}$ for all $t>0$ and $\varepsilon>0$. Thus,
\begin{align*}
    L_{p}(\alpha)\leq \dfrac{1}{\alpha}\inf\{P(-t\log(\psi))+t\alpha:t\in \R \},
\end{align*}and together with inequality \eqref{AimFOR} equality follows. Finally, from construction, it is straightforward to note that the domain of $L_{p}$ is  $[\log(p),\infty)=\bigcup_{n=1}^{\infty}[\log(p),n\log(p)]$. The fact that the infimum is attained in a unique $t_{\alpha}>0$ such that $\frac{d}{dt}P(-t\log(\psi))|_{t=t_{\alpha}}=-\alpha$ is a consequence of Theorem \ref{LyapunovSpectrumCompactAproximation} and that $\{\frac{d}{dt}P_{n}(-t\log(\psi))\}_{n\in \N}$ converge to $\frac{d}{dt}P(-t\log(\psi))$. \qed

\textit{Proof of Theorem \ref{PrincipalTheorem}} Let $\alpha\geq\log p $. Since $\frac{d}{dt}P(-t\log \psi )=-\log p \cdot \frac{p^{t}}{p^{t}-1}$, the value  $t_{\alpha}$ is the unique positive real number such that $\alpha=\log p\cdot \frac{p^{t_\alpha}}{p^{t_\alpha}-1}$, that is,
\begin{align*}
    t_{\alpha}=\log_{p}\left(\dfrac{\alpha}{\alpha-\log p}\right).
\end{align*} A direct computation then gives 
\begin{align*}
    P(-t_{\alpha}\log \psi)=\log(p-1)+\log(\alpha-\log p)-\log\log p.
\end{align*}Substituting these values into equation \eqref{EcuSInCalculoExplicito} yields precisely equation \eqref{Ecu2PrincipalTheorem}. \
\qed

\begin{example}
    If $x\in p\Z_p \setminus F$ is a periodic point of $T_{p}$ with period $n$, then there exist positive integers $a_{1},...,a_{n}$ such that $\lambda_{p}(x)=\log p\cdot \frac{a_{1}+\cdots +a_{n}}{n}=\log p\cdot A(\{a_{i}\})$, where $A(\{a_{i}\})$ denotes the arithmetic mean of the digits $\{a_{1},...,a_{n}\}$. 
If $x$ is a fixed point of $T_{p}$ in $I_{(a,b)}$, then $\lambda_{p}(x)=a\log p$, and Theorem \ref{PrincipalTheorem} yields
\begin{align*}
        L_{p}(a\log p)=\begin{cases}
            \dfrac{\log(p-1)+\log(a-1)}{a\log p}+\dfrac{\log\left(1+\frac{1}{a-1}\right)}{\log p} & \text{ if $a\neq 1$}\\
            \dfrac{\log(p-1)}{\log p} & \text{ if $a=1$}
        \end{cases}.
    \end{align*}If $x$ is periodic and $\lambda_{p}(x)=\log p\cdot A(\{a_{i}\})$ such that $a_{i}\neq a_{j}$ for some $1\leq i,j\leq n$, then
    \begin{align*}
        L_{p}\left(\log p\cdot A(\{a_{i}\})\right)= \dfrac{\log(p-1)+\log(A(\{a_{i}\})-1)}{A(\{a_{i}\})\log p}+\dfrac{\log\left(1+\frac{1}{A(\{a_{i}\})-1}\right)}{\log p}.
    \end{align*}
\end{example}
\begin{remark}
     It is worth noting that fixed points $x$ of $T_{p}$ satisfy 
    \begin{align}\label{PetalicNumbers}
        x^{2}+b_{1}(x)x-p^{a_{1}(x)}=0.
    \end{align} In contrast, fixed points of the Gauss map correspond to the so-called metallic means (see \cite{ds}). These are positive solutions of $y^{2}+ny-1=0$. For instance, when $n=1$, the golden ratio $g$ satisfies that $g-1$ is a fixed point of the Gauss map.
\end{remark}

\section{Rational approximation for continued fraction expansions}\label{geompotential}

 We now turn to the rational approximations obtained from the truncations in Schneider's continued fraction expansion. For each $n\in \N$ and $x\in p\Zp\setminus F$, we consider the limit
\begin{align}\label{EcuRationalApprox}
    -\lim_{n\to \infty}\dfrac{1}{n}\log \left|x-\dfrac{p_{n}(x)}{q_{n}(x)}\right|_{p},
\end{align}whenever it exists. Define $\psi_{2}:p\Z_{p}\setminus F\rightarrow \R$ by 
\begin{align*}
    \psi_{2}(x)=p^{-a_{1}(x)}\left|x-\dfrac{p^{a_{1}(x)}}{b_{1}(x)}\right|^{-1}_{p}.
\end{align*}Next, we introduce the real-valued function $\varphi:p\Z_{p}\setminus F\to \R$ defined by $\varphi(x)=p^{a_{2}(x)}\psi_{2}(T_{p}(x))/\psi_{2}(x)$.
\begin{lemma}\label{Lemma1}
    The limit in equation \eqref{EcuRationalApprox} coincides with the pointwise limit of Birkhoff's means for $\log \varphi$. Indeed,
    \begin{align*}
        -\lim_{n\to \infty}\dfrac{1}{n}\log \left|x-\dfrac{p_{n}(x)}{q_{n}(x)}\right|_{p}=\lim_{n\rightarrow \infty}\dfrac{1}{n}\sum^{n-1}_{k=0}\log\varphi\left(T_p^kx\right)
    \end{align*}
\end{lemma}
\begin{proof}
    Let $x\in p\Z_{p}\setminus F$. By definition, the second rational approximation satisfies
    \begin{align*}
        \psi_{2}(T_{p}(x))&=p^{-a_{1}(T_{p}(x))}\left|\dfrac{p^{a_{1}(T_{p}(x))}}{b_{1}(T_{p}(x))+T_{p}^{2}(x)}-\dfrac{p^{a_{1}(T_{p}(x))}}{b_{1}(T_{p}(x))}\right|_{p}^{-1}\\&=p^{-a_{2}(x)}\left|\dfrac{p^{a_{2}(x)}}{b_{2}(x)+T_{p}^{2}(x)}-\dfrac{p^{a_{2}(x)}}{b_{2}(x)}\right|_{p}^{-1},
    \end{align*} and by equation \eqref{Equation1} we have $\psi_{2}(T_{p}(x))=p^{-a_{2}(x)-a_{1}(x)}\left|x-\frac{p_{2}(x)}{q_{2}(x)}\right|_{p}$. Hence
    \begin{align*}
        \varphi(x)=\dfrac{\psi_{2}(T_{p}(x))}{\psi_{2}(x)}=\dfrac{\left|x-\dfrac{p_{2}(x)}{q_{2}(x)}\right|_{p}^{-1}}{\left|x-\dfrac{p_{1}(x)}{q_{1}(x)}\right|_{p}^{-1}}.
    \end{align*} Consequently,
    \begin{align*}
        \sum^{n-1}_{k=0}\log \varphi(x)&=-\sum^{n-1}_{k=0}\log\left|x-\dfrac{p_{k+1}(x)}{q_{k+1}(x)}\right|_{p}-\log\left|x-\dfrac{p_{k}(x)}{q_{k}(x)}\right|_{p}\\
        &= -\log\left|x-\dfrac{p_{n+1}(x)}{q_{n+1}(x)}\right|_{p}+\log\left|x-\dfrac{p_{1}(x)}{q_{1}(x)}\right|_{p}.
    \end{align*} Dividing by $n$ and taking limits yields
    \begin{align*}
        \lim_{n\rightarrow \infty}\dfrac{1}{n}\sum^{n-1}_{k=0}\log\varphi\left(T_p^kx\right)=\lim_{n\rightarrow \infty}-\dfrac{1}{n}\log\left|x-\dfrac{p_{n+1}(x)}{q_{n+1}(x)}\right|_{p}.
    \end{align*}
\end{proof}

\begin{proposition}\label{Proposition 52}
    The potential $\log \varphi$ is locally constant on cylinder sets of length 3. Moreover, for each $x\in p\Z_{p}\setminus F,$
   $$\log \varphi(x)=v_{p}\left(T^{2}_{p}(x)\right)\log p=a_{3}(x)\log(p).$$
\end{proposition}

\begin{proof}Observe that
    \begin{align*}
        \dfrac{\psi_{2}(T_p(x))}{\psi_{2}(x)} &=\dfrac{p^{-a_1\left(T_p(x)\right)}\left| T_p(x)-\dfrac{p^{a_1(T_p(x))}}{b_1(T_p(x))}\right|_p^{-1}}{p^{-a_1(x)}\left| x-\dfrac{p^{a_1(x)}}{b_1(x)} \right|_p^{-1}}\\
        &= p^{a_1(x)-a_1(T_p(x))} \dfrac{\left|\dfrac{b_1(T_p(x))T_p(x)-p^{a_1(T_p(x))}}{b_1(T_p(x))}\right|^{-1}_p}{\left|\dfrac{xb_1(x)-p^{a_1(x)}}{b_1(x)}\right|^{-1}_p}\\
        &=p^{a_1(x)-a_1(T_p(x))}\left|\dfrac{b_1(T_p(x))T_p(x)-p^{a_1(T_p(x))}}{xb_1(x)-p^{a_1(x)}} \right|_p^{-1}.
\end{align*} Using the definition of $T_{p}$, one finds $\left|\frac{p^{a_{1}(x)}}{b_{1}(x)+T_{p}(x)}b_{1}(x)-p^{a_{1}(x)}\right|_{p}=|T_{p}(x)|_{p}|x|_{p}$, so that
\begin{align*}
  \dfrac{\psi_{2}(T_p(x))}{\psi_{2}(x)}  &=p^{a_1(x)-a_2(x)} \left|\dfrac{T_p(x)T_p^2(x)}{xT_p(x)} \right|_p^{-1}\\
        &=p^{a_1(x)-a_2(x)}\cdot p^{v_p(T_p^2(x))-v_p(x)}\\
        &=p^{v_p(T_p^2(x))-v_p(T_p(x))}.
\end{align*}
Therefore, if $x$ and $y$ belong to the same cylinder of length 3 (that is, $a_{i}(x)=a_{i}(y)$ and $b_{i}(x)=b_{i}(y)$ for $1\leq i \leq 3$), then
    \begin{align*}
        \left|\log\varphi(x)-\log\varphi(y) \right|&=\left|\log\left(p^{a_2(x)}\dfrac{\psi_{2}(T_p(x))}{\psi_{2}(x)} \right) -\log\left(p^{a_2(y)}\dfrac{\psi_{2}(T_p(y))}{\psi_{2}(y)} \right)\right|\\
        &=|a_2(x)+v_p(T_p^2(x))-v_p(T_p(x))-a_2(y)-v_p(T_p^2(y))+v_p(T_p(x))|\\
        &=|v_p(T_p^2(x))-v_p(T_p^2(y))|\\
        &=|a_{3}(x)-a_{3}(y)|=0.
    \end{align*}
So $\log \varphi$ is locally constant on such cylinders, and $\log\varphi (x)=a_{3}(x)\log p$ for all $x\in p\Z_{p}\setminus F$. 
\end{proof}
Let $x\in J_{p}(\alpha)$. Since $J_{p}(\alpha)$ is a $T_{p}$-invariant set, we have $\alpha=\lambda_{p}(x)=\lambda_{p}(T_{p}^{2}(x))$. Then, by Lemma \ref{Lemma1} and Proposition \ref{Proposition 52}, 
\begin{align*}
    \alpha =\lim_{n\to \infty}\log p\cdot \dfrac{a_{1}(T_{p}^{2}(x))+\cdots+a_{n}(T_{p}^{2}(x))}{n}=\lim_{n\to \infty}\log p\cdot \dfrac{a_{3}(x)+\cdots+a_{3}(T_{p}^{n-1}(x))}{n}.
\end{align*} Hence
\begin{align*}
     \lambda_{p}(x)=\lim_{n\to \infty}\dfrac{1}{n}\sum^{n-1}_{k=0}\log\varphi(T_{p}^{k}(x))=-\lim_{n\to \infty}\dfrac{1}{n}\log \left|x-\dfrac{p_{n}(x)}{q_{n}(x)}\right|_{p}.
\end{align*}Thus, the Lyapunov exponent of $x$ coincides with the mean exponential rate at which $p$-adic rational approximations converge to $x$. Consequently, Theorem \ref{PrincipalTheorem} implies Corollary \ref{ApproximationSpectrumTHM}. Indeed, Theorem \ref{LyapunovSpectrumCompactAproximation} yields an analogous statement for each subsystem $(T_{p,n},p\Z_{p_,n}\setminus F)$. This allows us to characterize, in terms of dimension theory, sets of $p$-adic integers in $p\Z_{p,n}$ according to the rate at which their rational approximations converge.
\begin{theorem}
    Let $n\in \N$. For each $\alpha\in [\log p,n\log p]$ the Hausdorff dimension of the set of points $x$ in $p\Z_{p,n}\setminus F$ satisfying
    \begin{align*}
        -\lim_{m\to \infty}\dfrac{1}{m}\log \left|x-\dfrac{p_{m}(x)}{q_{m}(x)}\right|_{p}=\alpha,
    \end{align*} equals $\frac{1}{\alpha}P_{n}(-t_{\alpha}\log\psi)+t_{\alpha,n},$ where $P_{n}(-t_{\alpha,n}\log \psi)$ denotes the topological pressure of the potential $-t_{\alpha,n}\log \psi$ and $t_{\alpha,n}$ is the unique real number $t$ satisfying 
    \begin{align*}
        -\alpha=\log p\cdot \left(\dfrac{n}{p^{tn}-1}-\dfrac{p^{t}}{p^{t}-1}\right).
    \end{align*}
\end{theorem}

\section{Proof of technical results.}
\subsection{Proof of Lemma \ref{BowenFormula}}
 
\begin{proof}
    Let $\delta>0$ and $N_\delta=\min\left\{n\in \N:p^{-n}\leq \delta \right\}$. Then, the family
    $$\left\{I_{(a_i,b_i)}:(a_i,b_i)\in E_{n},\text{where } i\in \{1,...,N_\delta\} \right\}$$
    is an open and countable covering of $p\Z_{p,n}\setminus F$. For any $s>0$

    \begin{align*}
        \Hcali_\delta^s\left(p\Z_{p,n}\setminus F \right) \leq \sum_{(a_i,b_i)_{i=1}^{N_{\delta}}\in E_{n}^{N_\delta}}\left|I_{(a_i,b_i)_{i=1}^{N_{\delta}}} \right|_p^s = \sum_{(a_1,...,a_{N_\delta})\in \left\{1,2,...,n\right\}^{N_\delta}} (p-1)p^{-s\sum_{i=1}^{N_\delta}a_i},
    \end{align*} because each $a_{i}$ is considered $p-1$ times for each $b_{i}$. Due to each $a_{i}$ belongs to $\{1,...,n\}$, thus we obtain
    \begin{align*}
    \Hcali_\delta^s\left(p\Z_{p,n}\setminus F \right)     
        &\leq (p-1)^{N_\delta}\left(\sum_{k=1}^{n}p^{-ks} \right)^{N_\delta}.
    \end{align*}

Recalling that $\sum_{k=1}^n p^{-ks}\to 0$ as $s\to \infty$, and $\sum_{k=1}^{n}p^{-ks}\to n$ as $s\to 0$, then by continuity, there exists $\tilde s>0$ such that

\begin{equation}\label{Ecu2BowenFormular}
    \sum_{k=1}^n\dfrac{1}{p^{k\tilde s}}=\dfrac{1}{p-1}.
\end{equation}

In this way, for all $\delta>0$ 

$$\Hcali_\delta^{\tilde{s}}(p\Z_{p,n}\setminus F)\leq (p-1)^{N_\delta}\left( \sum_{k=1}^{N}p^{-k\tilde{s}}\right)^{N_\delta}=1.$$

We deduce that $\dim_{\mathrm{H}}(p\Z_{p,n}\setminus F)\leq \tilde{s}$. Furthermore, we conclude $\tilde{s}\leq 1$,  since 
\begin{align*}
    (p-1)^{N_\delta}\left(\dfrac{p^n-1}{p^n(p-1)} \right)^{N_\delta}=\left(1-\dfrac{1}{p^n}\right)^{N_\delta}\leq 1.
\end{align*} 
On the other hand, let $\Ucali$ be a countable covering such that $|\Ucali_i|_{p}<\delta.$ Since $p\Z_{p,n}$ is a compact space, we can consider a finite collection of elements in the covering. We suppose that the subcovering has $M$ elements. For each $1\leq i \leq M$, let $k_i$ the only positive integer such that 
$$p^{-(k_i+1)}\leq |\Ucali_i|_p\leq p^{-k_i}.$$
For each $k\in \N$ , we define the open cover

$$B_{k}=\left\{B\left( x_{(a_i,b_i)},p^{-k} \right):(a_i,b_i)_{i=1}^{N_{\delta}}\in E_{n}^{N_{\delta}}\right\},$$ 
where $ x_{(a_i,b_i)}$ denotes the periodic point in $I_{(a_i,b_i)_{i=1}^{N_{\delta}}}$ with period $N_{\delta}$.  If $j\geq k_i$, then $\Ucali_{i}$ intersects at most $(n\cdot(p-1))^{\lfloor \frac{j-k_{i}}{n}\rfloor}$ balls of $B_{j}$. By definition of $\tilde{s}$

\begin{align*}
\left( n(p-1)\right)^{j-k_i}=n^{j-k_i}(p-1)^{j}\left(\sum_{k=1}^n \dfrac{1}{p^{k\tilde s}} \right)^{k_i}\leq n^{j}(p-1)^j\left(\dfrac{1}{n}\sum_{k=0}^{n-1}p^{-k\tilde s} \right)^{k_i}p^{\tilde s}|\Ucali_i|_p^{\tilde s}.
\end{align*}
Furthermore, $p^{-k\tilde{s}}\leq 1$ since $\tilde{s}\geq 0$. Then
\begin{align*}
    (n\cdot (p-1))^{j-k_{i}}\leq (n(p-1))^{j}p^{\tilde{s}}|\Ucali_{i}|^{\tilde{s}}_{p}.
\end{align*}If we choose $j\geq \max\{k_{i}\}$, we bound by the number of balls $B_{j}$ that intersect each element in $\mathcal{U}$
\begin{align*}
    (n(p-1))^{j}&\leq \sum_{i=1}^{M}(n\cdot (p-1))^{j-k_{i}} \leq \sum_{i=1}^{M} (n(p-1))^{j}p^{\tilde{s}}|\Ucali_{i}|^{\tilde{s}}_{p}.
\end{align*}Therefore, $p^{-\tilde{s}}\leq \sum_{i=1}^{M}|\Ucali_{i}|_{p}^{\tilde{s}} $ which is independent of $\delta$, so $\mathcal{H}^{\tilde{s}}(p\Z_{p,n}\setminus F)\geq 0$. By definition of Hausdorff dimension $\tilde{s}\leq \dim_{\mathrm{H}}\left(p\Z_{p,n}\setminus F\right)$. Finally, computing pressure as in equation \eqref{EcuPressure} we get
\begin{align*}
    P_n(-\tilde s \log \varphi)&=\lim_{m\to \infty} \dfrac{1}{m}\log \left(\sum_{\substack{T_{p,n}^m x=x}} e^{-\tilde{s}S_m \log \varphi(x)} \right)\\
    &=\lim_{m\to \infty}\dfrac{1}{m}\log \left( \sum_{\substack{(a_i,b_i)_{i=1}^m \in E_{n}^{m}}}\prod_{k=1}^{m}p^{-\tilde{s}a_{k}}\right)\\
    &= \lim_{m\to \infty}\dfrac{1}{m}\log\left((p-1)^{m}
    \left(\sum^{n}_{k=1}p^{-\tilde{s}k}\right)^{m}\right),
\end{align*}where the last expression is equal to zero by equation \eqref{Ecu2BowenFormular}.

\end{proof}

\subsection{Proof of Theorem \ref{PrincipalPesinWeiss}}

 We recall that a cover $\mathcal{V}$ of $p\Z_{p,n}\setminus F$ is said to be a \textit{Moran cover} if there exists $M>0$ such that for each $x\in p\Z_{p,n}\setminus F$ and any $r>0$ we have 
\begin{align*}
    \#\{V_{i}\in \mathcal{V}: V_{i}\cap B(x,r)\neq \emptyset \}\leq M.
\end{align*} Let $x\in p\Z_{p,n}\setminus F$. Since, $\psi(x)=p^{a_{1}(x)}$, then for each $m\in \N$
\begin{align}\label{Ecu1}
    0\leq \left(\prod^{m-1}_{k=0}\psi(a_{1+k}(x)) \right)^{-1}\leq p^{-m}.
\end{align}Given any element $w\in \Sigma_{n}$ and $r>0$, there exists a unique $m(w)\in \N$ such that 
\begin{align}\label{Ecu2}
    \prod^{m(w)-1}_{k=0}p^{-a_{1+k}(\pi_{n}(w))}>r>\prod^{m(w)}_{k=0}p^{-a_{1+k}(\pi_{n}(w))}.
\end{align}From inequality \eqref{Ecu1} we notice $m(w)\to \infty$ as $r\to 0$. Fix $w\in \Sigma_{n}$ and consider the cylinder $$C(w)=[(a_{i}(\pi_{n}(w)),b_{i}(\pi_{n}(w))):1\leq i \leq m(w)],$$ where $m(w)$ satisfies inequality \eqref{Ecu2}. Note that if $w''\in C(w)$, then for each $0\leq k \leq m(w)-1$ we have $a_{1+k}(\pi_{n}(w''))=a_{1+k}(\pi_{n}(w))$. Thus, $m(w'')=m(w)$ and so $C(w)=C(w'')$ . Furthermore, the collection $\{C(w)\}_{w\in \Sigma_{n}}$ is a partition of $\Sigma_{n}$. On the other hand, by inequality \eqref{Ecu1} there exists $N_{r}\in \N$ such that for each $w\in \Sigma_{n}$ we have $m(w)\leq N_{r}$. So, the cover $\mathcal{C}_{r}=\{C(w)\}_{w\in \Sigma_{n}}$ is a finite and open cover of $\Sigma_{n}$. Thus, we write $\mathcal{C}_r$ as $\{C_r^{j}\}_{j=1}^{M_r}$, for some $M_r\in \N$. We define $R_r^{j}$ as the image of $C_r^{j}$ via $\pi_n$, that is 
\begin{align*}
    \mathcal{R}_{r}=\pi_{n}(\mathcal{C}_{r})=\{R_{r}^{i}\}_{i=1}^{M_{r}}
\end{align*}is a finite and open cover of $p\Z_{p,n}\setminus F$. 
\begin{lemma}\label{MoranCover}
    Let $r>0$. Then $\mathcal{R}_{r}$ is a Moran cover of $p\Z_{p,n}\setminus F$.
\end{lemma}
\begin{proof}
     By definition of the partition $\mathcal{R}_{r}$, for each $R_{r}^{j}$ and for any $x\in  R_{r}^{j}$
    \begin{align*}
        R_{r}^{j}=B\left(x,\prod_{k=0}^{m(\pi_{n}^{-1}(x))}p^{-a_{1+k}(x)}\right)\subseteq B(x,r) \subseteq B\left(x,\prod_{k=0}^{m(\pi_{n}^{-1}(x))-1}p^{-a_{1+k}(x)}\right).
    \end{align*} Now let $x\in p\Z_{p,n}\setminus F $. If $R_{r}^{j}\cap B(x,r)\neq \emptyset$ for some $R_{r}^{j}\in \mathcal{R}_{r}$ and $y\in R_{r}^{j}\cap B(x,r)$,  by the ultrametric property of the $p$-adic metric 
    \begin{align*}
        B\left(y,\prod_{k=0}^{m(\pi_{n}^{-1}(y))}p^{-a_{1+k}(y)}\right)=R_{r}^{j}.
    \end{align*}Therefore $B(y,r)=B(x,r)$, and so 
    \begin{align*}
        B\left(y,\prod_{k=0}^{m(\pi_{n}^{-1}(y))}p^{-a_{3+k}(y)}\right)\subseteq B(x,r)\subseteq B\left(y,\prod_{k=0}^{m(\pi_{n}^{-1}(y))-1}p^{-a_{3+k}(y)}\right).
    \end{align*}Moreover, for each $1\leq k\leq m(\pi_{n}^{-1}(x))-1$ we have that $a_{k}(x)=a_{k}(y)$ and $b_{k}(x)=b_{k}(y)$. Since $(a_{k}(y),b_{k}(y))\in E_{n}$, then we obtain
    \begin{align*}
        \#\{R_{r}^{j}\in \mathcal{R}_{r}: R_{r}^{j}\cap B_{p}(x,r)\neq \emptyset \}\leq n(p-1).
    \end{align*}
\end{proof}

 Let $0<r<1$. Let $\mathfrak{m}$ be the equilibrium state of $-\tilde{s}\log \psi$. Recall that $\mu_{q}$ is the equilibrium state of $\phi_{q,n}$ and $\nu_{q}=(\pi_{n})_{*}\mu_{q}$.  Moreover, $\mathfrak{m}$, $\nu_{q}$ and $\nu_{\rho}$ are Gibbs measures, thus for each $q\in \R$ there exist positive constants $C_{1}$ and $C_{2}$ such that for each set $R^{j}_{r}$ (see \cite[Lemma 1]{yh})
    \begin{align}\label{Ecu6}
        C_{1}\leq \dfrac{\nu_{q}(R^{j}_{r})}{\mathfrak{m}(R^{j}_{r})^{T(q)\tilde{s}}\cdot \nu_{\rho}(R^{j}_{r})}\leq C_{2}.
    \end{align} For each $\alpha\geq 0$ we define 
\begin{align*}
    \tilde{K}_{\alpha}=\left\{w\in \Sigma_{n}:\lim_{m\to \infty}\dfrac{\sum^{m-1}_{k=0}\log(\vartheta\circ \sigma^{k}(w))}{\sum^{m-1}_{k=0}\log(\psi\circ \pi_{n}(\sigma^{k}(w)))}=\alpha\right\},
\end{align*}and the corresponding spectrum $\tilde{f}_{\nu_{\rho}}(\alpha)=\dim_{\mathrm{H}}\left(\tilde{K}_{\alpha}\right)$. For each $q\in \R$, let
\begin{align*}
    \alpha(q)=\dfrac{\int_{\Sigma_{n}}\log \vartheta(w) \ d\mu_{q}}{\int_{\Sigma_{n}}\log(\varphi\circ \pi_{n}(w))d\mu_{q}}.
\end{align*} 
\begin{lemma}\label{Lemma3}
    For each $q\in \R$ we obtain
    \begin{enumerate}
        \item $\nu_{q}(\pi_{n}(\tilde{K}_{\alpha(q)}))=1$.
        \item $d_{\nu_{q}}(x)=\mathcal{T}(q)+q\alpha(q)$ for $\nu_{q}$-almost every point $x\in \pi_{n}(\tilde{K}_{\alpha(q)})$.
        \item $\dim_{\mathrm{H}}\pi_{n}(\tilde{K}_{\alpha(q)})=\mathcal{T}(q)+q\alpha(q)$.
    \end{enumerate}
\end{lemma}
\begin{proof}Item 1 is a consequence of Birkhoff's ergodic theorem since $\mu_{q}$ is ergodic, i.e., for $\mu_{q}$-almost every point $w\in \Sigma_{n}$
\begin{align*}
    \lim_{m\to \infty}\dfrac{\sum^{m-1}_{k=0}\log(\vartheta\circ \sigma^{k}(w))}{\sum^{m-1}_{k=0}\log(\varphi \circ \pi_{n}(\sigma^{k}(w)))}=\lim_{m\to \infty }\dfrac{S_{m}\log(\vartheta(w))}{S_{m}\log (\varphi \circ \pi_{n}(w))}=\alpha(q).
\end{align*} Furthermore, fixed $\varepsilon>0$, for each $w\in \tilde{K}_{\alpha(q)}$ there exists a positive integer $N(w)$ such that for all $n>N(w)$
\begin{align}\label{Ecu8}
    \left|\dfrac{\sum^{n-1}_{k=0}\log(\vartheta\circ \sigma^{k}(w))}{\sum^{n-1}_{k=0}\log(\varphi \circ \pi_{n}(\sigma^{k}(w)))}-\alpha(q)\right|\leq \varepsilon.
\end{align} We truncate elements in $\tilde{K}_{\alpha(q)}$ depending on $N(w)$. Denote $Q_{l}=\{w\in \tilde{K}_{\alpha(q)}:N(w)\leq l \}$ for each $l>0$.  These sets are concatenated by inclusion $Q_{l}\subset Q_{l+1}$ and $\tilde{K}_{\alpha(q)}=\bigcup^{\infty}_{l=1}Q_{l}$. Therefore, there exists some $l_{0}>0$ such that $\mu_{q}(   Q_{l})>0$ for all $l \geq l_{0}$.

Given $0<r<1$, consider now the Moran cover $\{C^{j}_{r,l}\}$ of $Q_{l}$. Thus, for each $j$ there exists $w_{j}\in Q_{l}\cap C_{r,l}^{j}$ such that $C_{r,l}^{j}=C(w_{j})$. Moreover, if $r$ is small enough, we can assume that $m(w_{j})\geq l$ for all $j$. On the other hand, since $\mu_{q}$ is a Gibbs measure, for each $w\in \Sigma_{n}$ and $m>0$
\begin{align}\label{Ecu9}
    C_{9}\leq  \dfrac{\mu_{q}(C(w))}{\prod^{m-1}_{k=0}|\varphi  (\pi_{n}\circ \sigma^{k}(w))|^{-\mathcal{T}(q)}(\vartheta(\sigma^{k}(w)))^{q}}\leq C_{10},
\end{align}where $C_{9}>0$ and $C_{10}>0$. By inequalities \eqref{Ecu8} and \eqref{Ecu9}, and the boundedness property of the Moran cover, we get for all $m\geq l$ and $x\in \pi_{n}(Q_{l})$
\begin{align*}
    \nu_{q}(B(x,r)\cap \pi_{n}(Q_{l})) & \leq \sum^{N(x,l)}_{i=1}\mu_{q}(C^{j_{i}}_{r,l}),
    \end{align*} where $N(x,l)\leq n(p-1)$ is the cardinality of $\{C_{r,l}^{j}:C_{r,l}^{j}\cap B(x,r)\neq \emptyset\}$. Thus
    \begin{align*}
   \nu_{q}(B(x,r)\cap \pi_{n}(Q_{l}))  & \leq  C_{10}\sum^{N(x,l)}_{i=1}\prod^{m-1}_{k=0}|\psi(\pi_{n}\circ \sigma^{k}(w_{i}))|^{-\mathcal{T}(q)}(\vartheta(\sigma^{k}(w_{i})))^{q}\\
    & \leq C_{10}\sum^{N(x,l)}_{i=1}\prod^{m-1}_{k=0}|\psi(\pi_{n}\circ \sigma^{k}(w_{i}))|^{-\mathcal{T}(q)-q(\sigma(q)-\varepsilon)}\\
    & \leq C_{11}\cdot r^{\mathcal{T}(q)+q(\alpha(q)-\varepsilon)},
\end{align*}where $w_{i}\in C_{r,l}^{j_{i}}$ in each sumand and $C_{11}>0$. Recall $\nu_{q}=(\pi_{n})_{*} \mu_{q}$, thus for $\nu_{q}$-a.e.p. $x\in \pi_{n}(Q_{l})$ 
\begin{align*}
    \nu_{q}(B(x,r))\leq \nu_{q}(B(x,r)\cap \pi_{n}(Q_{l})),
\end{align*} since $B(x,r)\subseteq \pi_{n}(Q_{l})$. Then, for any $l>l_{0}$ and $\nu_{q}$-a.e.p. $x\in \pi_{n}(Q_{l})$ 
\begin{align*}
    \underline{d}_{\nu_{q}}(x)& = \liminf_{r\to 0}\dfrac{\log \nu_{q}(B(x,r))}{\log r}\\
    & \geq \liminf_{r\to 0}\dfrac{\log(\nu_{q}(B(x,r)\cap \pi_{n}(Q_{l})))}{\log r}\\
   & \geq \mathcal{T}(q)+q(\alpha(q)-\varepsilon).
\end{align*} Since the sets $Q_{l}$ are nested, for $\nu_{q}$-a.e.p. $x\in \pi_{n}(\tilde{K}_{\alpha(q)})$ 
\begin{align*}
    \underline{d}_{\nu_{q}}(x)\geq\mathcal{T}(q)+q(\alpha(q)-\varepsilon).
\end{align*} On the other hand, let $0<r<1$ be fixed. For each $w\in Q_{l}$ recall that we have
\begin{align*}
    R_{r}(\pi_{n}(w))\subseteq B(x,r) \subseteq B(x,2r),
\end{align*}since $R_{r}(\pi_{n}(w))$ intersects at most $n(p-1)$ balls of radius $r$ around $x$. Thus, by equations \eqref{Ecu8} and \eqref{Ecu9} for every $x\in Q_{l}$ we get
\begin{align*}
    v_{q}(B(x,2r))&\geq v_{q}(R_{r}(\pi_{n}(w)))\\
    & \geq C_{9}\prod^{m(w)}_{k=0}\psi(\pi_{n}\circ \sigma^{k}(w))^{-\mathcal{T}(q)}(\vartheta(\sigma^{k}(w)))^{q}\\
    &\geq C_{9}\prod^{m(w)}_{k=0}\psi(\pi_{n}\circ\sigma^{k}(w))^{-\mathcal{T}(q)-q(\alpha(q)+\varepsilon)}\\
    & \geq C_{9}r^{\mathcal{T}(q)+q(\alpha(q)+\varepsilon)}.
\end{align*} Thus,
\begin{align*}
    \overline{d}_{\nu_{q}}(x)&=\limsup_{r\to 0}\dfrac{\log \nu_{q}(B(x,2r))}{\log 2r}\\
    & \leq \limsup_{r\to 0}\dfrac{\log C_{9}r^{\mathcal{T}(q)+q(\alpha(q)+\varepsilon)}}{\log(2r)}\\
    & \leq \mathcal{T}(q)+q(\alpha(q)+\varepsilon).
\end{align*} This implies the second statement of the lemma. Furthermore, 
\begin{align*}
    \dim_{\mathrm{H}}(\tilde{K}_{\alpha(q)})\geq \mathcal{T}(q)+q\alpha(q),
\end{align*}and, by \cite[Proposition 2.3]{fa} we get $\dim_{\mathrm{H}}(\tilde{K}_{\alpha(q)})\leq \mathcal{T}(q)+q\alpha(q)$.\end{proof}
\noindent Now we show the pointwise dimension from the symbolic model coincides with the one in $p\Z_{p,n}\setminus F$.
\begin{proposition}\label{Proposition 1}
\hfill    \begin{enumerate}
        \item For each $q\in \mathbb{R}$, $w\in \tilde{K}_{\alpha(q)}$, and $x\in \pi_{n}(w)$, we have $d_{\nu_{q}}(x)=\alpha(q)$.
        \item For each $q\in \mathbb{R}$ and each $x\in K_{\alpha(q)}$ there exists $w\in \tilde{K}_{\alpha(q)}$ such that $\pi_{n}(w)=x$.  In other words, for each $q\in \mathbb{R}$ we have $\pi_{n}(\tilde{K}_{\alpha(q)})=K_{\alpha(q)}$. 
    \end{enumerate}
\end{proposition}
\begin{proof}
     Recall first that for $x,y\in  p\Z_{p,n}\setminus F$ such that $|x-y|_{p}<r_{0}$ with  $r_{0}=p^{-(n+1)}$ we have 
     \begin{align*}
         |T_{p,n}(x)-T_{p,n}(y)|_{p}\geq p|x-y|_{p}.
     \end{align*}Indeed, for any $0<r<r_{0}$, there exist $N(r)>0$ and positive constants $C_{12}$ and $C_{13}$ such that if $0\leq m \leq N(r)$, then for all $x\in  p\Z_{p,n}\setminus F$
     \begin{align}\label{Ecu10}
         C_{12}r\prod^{m-1}_{k=0}\psi(T_{p,n}^{k}(x))\leq  |T_{p,n}^{m}(B(x,r))|_{p}\leq C_{13}r\prod^{m-1}_{k=0}\psi(T^{k}_{p,n}(x)).
     \end{align} This follows from
     \begin{align*}
         |T_{p,n}^{m}(B(x,r))|_{p}\geq p^{a_{1}+\cdots +a_{m}}r\geq p^{m}r\dfrac{\prod^{m-1}_{k=0}\psi(T_{p,n}^{k}(x))}{\prod^{m-1}_{k=0}\psi(T_{p,n}^{k}(x))}\geq C_{12}r\prod^{m-1}_{k=0}\psi(T_{p,n}^{k}(x)),
     \end{align*}where $C_{12}=p^{m(1-n)}$, and
     \begin{align*}
         |T_{p,n}^{m}(B(x,r))|_{p}\leq 1\cdot \dfrac{r\prod^{m-1}_{k=0}\psi(T_{p,n}^{k}(x))}{r\prod^{m-1}_{k=0}\psi(T_{p,n}^{k}(x))}\leq C_{13}r\prod^{m-1}_{k=0}\psi(T_{p,n}^{k}(x)),
     \end{align*} where $C_{13}=r^{-1}p^{-m}$.  Analogously, there are positive constants $C_{14}$ and $C_{15}$ such that for all $x\in p\Z_{p,n}\setminus F$ and all $(a,b)\in E_{n}$ 
\begin{align}\label{Ecu11}
         C_{14}r\prod^{m-1}_{k=0}\psi(T_{p,n}^{k}(x))^{-1}\leq  |h_{a,b}^{m}(B(x,r))|_{p}\leq C_{15}r\prod^{m-1}_{k=0}\psi(T^{k}_{p,n}(x))^{-1},
     \end{align} where $h_{a,b}$  is the inverse branch of  $T^{-m}_{p,n}$ with image in $I_{a,b}$. Let fix $x\in p\Z p_{n}\setminus F$ and define $N=N(x,r)$ such that 
     \begin{align*}
         C_{13}r\prod^{N-1}_{k=0}\psi(T^{k}_{p,n}(x))\leq r_{0}\leq C_{13}r\prod^{N}_{k=0}\psi(T^{k}_{p,n}(x)).
     \end{align*} So $N(x,r)$ is the number of necessary iterations of $B(x,r)$ in order to growth and still having diameter lower than $r_{0}$. Thus, by equations \eqref{Ecu10} and \eqref{Ecu11} we obtain that
\begin{align*}
    \left|T_{p,n}^{N}(B(x,r))\right|_{p}\leq C_{13}r\prod^{N-1}_{k=0}\psi(T^{k}_{p,n}(x))\leq r_{0}=\left|B(T^{N}_{p,n}(x),r_{0})\right|_{p}.
\end{align*}So $B(x,r)\subset h(B(T^{N}_{p,n}(x),r_{0}))$ where $h$ is the inverse branch of $T^{-N}_{p,n}$ corresponding to $x$. Also, from equation \eqref{Ecu11}
\begin{align*}
    \left|h(B(T^{N}_{p,n}(x),r_{0}))\right|_{p}\leq C_{15}r_{0}\prod^{N-1}_{k=0}\psi(T^{k}_{p,n}(x))^{-1}\leq C_{16}r=|B(x,C_{16}r)|_{p},
\end{align*}where $C_{16}>0$ is a constant. Thus,
\begin{align*}
    B(x,r)\subset h(B(S^{N}(x),r_{0})) \subset B(x,C_{16}r).
\end{align*} Recall there exists $M_{r_{0}}\in \mathbb{N}$ such that $\mathcal{R}(T^{N}_{p,n}(x))=h_{0}(I_{a,b}),$ where $h_{0}$ is the corresponding branch of $T_{p}^{-M_{r_{0}}}$ \textcolor{red}{ } at $I_{(a,b)}$.  There are positive constants $C_{17}$ and $C_{18}$ such that 
\begin{align*}
    B(T^{N}_{p,n}(x),C_{17}r_{0})\subset \mathcal{R}(T^{N}_{p,n}(x))\subset B(T^{N}_{p,n}(x),C_{18}r_{0}).
\end{align*} This implies 
\begin{align*}
    h(B(T^{N}_{p,n}(x),C_{17}r_{0}))\subset h(\mathcal{R}(T^{N}_{p,n}(x)))\subset h(B(T^{N}_{p,n}(x),C_{18}r_{0})), 
\end{align*}and so 
\begin{align}\label{Ecu12}
    \nu_{q}(B(T^{N}_{p,n}(x),C_{17}r_{0}))\subset \nu_{q}(\mathcal{R}(T^{N}_{p,n}(x)))\subset \nu_{q}(B(T^{N}_{p,n}(x),C_{18}r_{0})).
\end{align} On the other hand, from Lemma \ref{Lemma3} and the fact that $\nu_{q}$ is a Gibbs measure, we obtain that is diametrically regular. Indeed, for $C_{19}=C^{2}p^{-1}e^{2P(\rho)}$ we get for any $r_{1}>0$ and $y\in p\Z_{p,n}\setminus F$ that
\begin{align*}
    \nu_{q}(B(r,pr_{1}))\leq C_{19}\nu_{q}(B(x,r_{1})).
\end{align*}Thus
\begin{align*}
    \nu_{q}(B(T^{N}_{p,n}(x),C_{18}r_{0}))&\leq C_{20}\nu_{q}(B(T^{N}_{p,n}(x),r_{0}))=C_{17}\nu_{q}(h(B(T^{N}_{p,n}(x),r_{0})))\\
    & \leq C_{20}\nu_{q}(B(x,C_{16}r))\\
    &\leq C_{21}\nu_{q}(B(x,r)),
\end{align*}where $C_{20}>0$ and $C_{21}>0$ are constants. Similarly
\begin{align*}
    \nu_{q}(B(T^{N}_{p,n}(x),C_{17}r_{0}))\geq C_{22}\nu_{q}(B(T^{N}_{p,n}(x),r_{0}))=C_{22}\nu_{q}(h(B(T^{N}_{p,n}(x),r_{0})))\geq C_{23}\nu_{q}(B(x,r)),
\end{align*}where $C_{22}>0$ and $C_{23}>0$. Thus, by equation \eqref{Ecu12} 
\begin{align*}
    C_{23}\nu_{q}(B(x,r))\leq \nu_{q}(h(R(T^{N}_{p,n}(x)))) \leq  C_{21}\nu_{q}(B(x,r)).
\end{align*}Since $\nu_{q}$ is a Gibbs measure for $\log \vartheta$ and $h(R(T^{N}_{p,n}(x)))$ is an element of the Moran cover, we get
\begin{align*}
    d_{v}(x)=\lim_{r\to 0}\dfrac{\nu_{q}(B(x,r))}{\log r}=\lim_{N(r)\to \infty}\dfrac{\log\prod^{N(r)-1}_{k=0}\vartheta(\sigma^{k}(w))}{\log \prod^{N(r)-1}_{k=0}\psi (\pi_{n}\circ \sigma^{k}(w))^{-1}},
\end{align*}where $\pi_{n}(w)=x$. Now assume that $d_{v}(x)=\alpha(q)$, we need to show the existence of a subsequential limit $N(r)\to \infty$ that implies the limit for $n\to \infty$.  Consider $r_{k}=p^{-k}$, it follows from the definition of $N(r)$ and the uniformly expansiveness of $T^{N}_{p,n}$
\begin{align*}
    N(r_{k+1})-N(r_{k})=\begin{cases}
        0 & \text{if $N(x,r_{k+1})=N(x,r_{k})$}\\
        1 & \text{if $N(x,r_{k+1})>N(x,r_{k})$}
    \end{cases}.
\end{align*}
And so the proposition follows. 
\end{proof}

 \textit{Proof of Theorem \ref{PrincipalPesinWeiss}.} 
For each $x\in p\Z_{p,n}\setminus F$ let $N(x,r)$ be the number of elements in $\Rcali$ whose intersection with $B(x,r)$ is non-empty. Recall from Lemma \ref{MoranCover} that $\Rcali=\{\Rcali_{r}^{j}\}$ is a Moran cover of $p\Z_{p,n}\setminus F$, thus $N(x,r)\leq n(p-1)$ for all $x\in p\Z_{p,n}\setminus F$. Since $\mathfrak{m}$ is a Gibbs measure, there exists a constant $C_{4}>0$ such that 
    \begin{align*}
        C_{4}^{-1}\leq \dfrac{\mathfrak{m}(\Rcali^{j}_{r})}{\left(\prod^{n-1}_{k=0}\psi\circ T_{p,n}^{k}(x)\right)^{-\tilde{s}}}\leq C_{4}.
    \end{align*}From the construction of $\Rcali$ (see equation \eqref{Ecu2}) we obtain 
    \begin{align}\label{Ecu7}
        C_{4}^{-1}r^{\tilde{s}}\leq \mathfrak{m}(\Rcali^{j}_{r})\leq C_{5}r^{\tilde{s}}, 
    \end{align}where $C_{5}=C_{4}p^{-n}$. On the other hand, given that $\Rcali$ is a disjoint cover of $p\Z_{p,n}\setminus F$ we obtain 
    \begin{align*}
        \sum_{\Rcali_{r}^{j}}\nu_{q}(\Rcali_{r}^{j})=1.
    \end{align*}Summing on each element of $\Rcali$ and using \eqref{Ecu6}
\begin{align*}
    1=\sum_{\Rcali_{r}^{j}}\nu_{q}(\Rcali^{j}_{r})&\leq C_{2}\sum_{\Rcali^{j}_{r}}\mathfrak{m}(\Rcali_{r}^{j})^{\mathcal{T}(q)/\tilde{s}}\cdot \nu_{\rho}(\Rcali^{j}_{r})^{q}\\
    & \leq C_{2} \sum_{\Rcali_{r}^{j}}C_{5}^{\mathcal{T}(q)/\tilde{s}}r^{T(q)}\nu_{\rho}(\Rcali_{r}^{j})^{q}\\
  &  = C_{5}^{\mathcal{T}(q)/\tilde{s}}C_{2}r^{\mathcal{T}(q)}\sum_{\Rcali_{r}^{j}}\nu_{\rho}(\Rcali^{j}_{r})^{q}.
\end{align*}Defining $C_{8}=C_{4}^{-\mathcal{T}(q)/\tilde{s}}C_{2}^{-1}$ we obtain $C_{8}\leq r^{\mathcal{T}(q)}\sum_{\Rcali_{r}}\nu_{\rho}(\Rcali_{r}^{j})^{q}$. Analogously, we obtain $C_{7}>0$ such that
\begin{align*}
    C_{8}\leq r^{\mathcal{T}(q)}\sum_{\Rcali_{r}}\nu_{\rho}(\Rcali_{r}^{j})^{q}\leq C_{7}.
\end{align*}Taking logarithms and dividing by $\log r$ we get that 
\begin{align}\label{Ecuu5}
    \lim_{r\to 0}\dfrac{\log\left(\sum_{\Rcali_{r}^{j}}\nu_{\rho}(R^{j}_{r})^{q}\right)}{\log r}=-\mathcal{T}(q).
\end{align}
The first statement of the theorem follows from the Lemma \ref{Lemma3} and proposition \ref{Proposition 1}. 

Recall that $\nu_{q}(K_{\alpha(q)})=1$ and that $\underline{d}_{\nu_{q}}(x)=\overline{d}_{\nu_{q}}(x)=\mathcal{T}(q)+q\alpha(q)$  for $\nu_{q}$-a.e.p. $x\in K_{\alpha(q)}$ by Lemma \ref{Lemma3}. Since $(p\Zp,|\cdot|_{p})$ is a complete separable metric space with finite topological dimension, then \cite[Proposition 3]{yh} implies that $\dim_{\mathrm{H}}(K_{\alpha(q)})=\mathcal{T}(q)+q\alpha(q)$. Statements 2 and 3 now follow with classical methods from thermodynamic formalism using the symbolic model (see \cite{yh,sa2}). This is, applying the implicit function theorem to $(q,r)\mapsto P( \phi_{q,r})$ to show that $\mathcal{T}(q)$ is real analytic \cite[Lemma 3]{yh}. Moreover, for all $q\in \R$  the function $\mathcal{T}(q)$ is convex, and strictly convex if and only if $\nu_{q}\neq \mathfrak{m}$ (see \cite[Lemma 5]{yh}). Second, $\alpha(q)=-\mathcal{T}'(q)$ (see \cite[Lemma 4]{yh}).  In particular, since $\alpha'(q)=-\mathcal{T}''(q)>0$, then $\alpha(q)$ has a non-singleton interval set as range. 

Now we show the last claim of the theorem. Let $r>0$ and consider the Moran cover $\{R_{r}^{j}\}$ of $p\Z_{p,n}\setminus F$. There are constants $C_{23}>0$ and $C_{22}>0$ such that  for all $x\in R_{r}^{j}$ 
\begin{align}\label{Ecu14}
    B(x,C_{25}r)\subset R_{r}^{j}\subset B(x,C_{26}r).
\end{align}Fixed one $x_{j}$ for each $R_{j}^{r}$. Since $\nu_{q}$ is diametrically regular, it follows that 
\begin{align}\label{Ecu15}
    \sum_{j}\nu_{q}(x_{j},C_{26}r)^{q}\leq C_{27}\sum_{j}\nu_{q}(B(x_{j},C_{25}r))^{q}\leq C_{27}\sum_{j}\nu_{q}(R^{j}_{r})^{q},
\end{align}where $C_{27}>0$ is a constant independent from $j$ and $r$. Now, consider $\mathcal{G}_{r}$ be a countable cover of $p\Z_{p,n}\setminus F$ by balls of radius $r$. For each $j\geq 0$ there exists $B(y_{i},r)\in \mathcal{G}_{r}$ such that $B(y_{j},r)\cap R^{j}_{r}\neq \emptyset$. Therefore, by equation \eqref{Ecuu5}
\begin{align}\label{Ecuu6}
    \mathcal{T}(q)\geq  -\lim_{r\to 0}\dfrac{\log(\sum_{B\in \mathcal{G}_{r}}\nu_{q}(B)^{q})}{\log r}.
\end{align} On the other hand, consider now another open cover $\mathcal{G}_{2C_{26}r}$ of $p\Z_{p,n}\setminus F$ by balls of radius $2C_{26}r$. By \eqref{Ecu14}, each set $R_{r}^{j}$ is in at least one element of $\mathcal{G}_{2C_{26}r}$.  We define an equivalence relation among the set of the Moran cover and say that $R_{r}^{i}$ and $R_{r}^{j}$ are equivalent if they are in the same element of $\mathcal{G}_{2C_{26}r}$. Thus, for each equivalent class $\xi_{k}$ represented by $B_{k}\in \mathcal{G}_{2C_{26}r}$ we have from \eqref{Ecu14} for each $q>0$
\begin{align}\label{Ecu16}
  \sum_{R_{r}^{j}\in \xi_{k}}\nu_{q}(R_{r}^{j})^{q}\leq n(p-1)\nu_{q}(B_{k})^{q}.  
\end{align}Using again that $\nu_{q}$ is diametrically regular we get for each $q<0$
\begin{align}\label{Ecu17}
    \sum_{R_{r}^{j}\in \xi_{k}}\nu_{q}(R_{r}^{j})^{q}\leq C_{28}\nu_{q}(B_{k})^{q},
\end{align}where $C_{28}>0$. And so, for any $q\in \mathbb{R}$ equations \eqref{Ecu16}
 and \eqref{Ecu17} imply
 \begin{align}\label{Ecu18}
     \sum_{R_{r}^{j}}\nu_{q}(R_{r}^{j})^{q}\leq C_{29}\sum_{B_{k}\in \mathcal{G}_{2C_{26}r}}\nu_{q}(B_{k})^{q} \leq C_{30}\sum_{B\in \mathcal{G}_{r}}\nu_{q}(B)^{q},
 \end{align} where $C_{29}>0$ and $C_{30}>0$ are constants. Therefore, the upper inequality of \eqref{Ecuu6} follows from equation \eqref{Ecuu5} and inequality \eqref{Ecu18}.

\qed

\section*{Acknowledgements}
The first author was partially supported by ANID Doctorado Nacional 21200910. The second author
was partially supported by ANID Doctorado Nacional 21210037. We thank Godofredo Iommi for valuable comments and for having read this manuscript. We are also grateful to
Patricio Perez-Pi\~{n}a and Till Hauser for fruitful discussions.  

\bibliographystyle{amsalpha}
\bibliography{refs.bib}

\end{document}